\documentclass[a4paper, 12pt, oneside, notitlepage]{amsart}
\usepackage{amsmath,amssymb,amsthm,graphicx,mathrsfs,bbm,url}
\usepackage[margin=3cm]{geometry}
\usepackage{amsthm}
\usepackage{stmaryrd}
\usepackage{wrapfig}
\usepackage{enumitem}
\usepackage{mathtools}
\usepackage[all]{xy}
\usepackage[utf8]{inputenc}
\usepackage[usenames,dvipsnames]{color}
\usepackage[colorlinks=true,linkcolor=Red,citecolor=Green]{hyperref}
\usepackage[super]{nth}
\usepackage[open, openlevel=2, depth=3, atend]{bookmark}
\hypersetup{pdfstartview=XYZ}
\usepackage[font=footnotesize]{caption}
\usepackage{subcaption}
\usepackage{caption}
\usepackage{tikz}
\usepackage{setspace}
\usepackage[normalem]{ ulem }
\usepackage{soul}
\usetikzlibrary{decorations.pathreplacing}
\usetikzlibrary{arrows}
\usetikzlibrary{shapes.misc}
\usetikzlibrary{shapes.symbols}
\usetikzlibrary{patterns}
\usepackage{epstopdf}
 
\usepackage{hyperref}
\usepackage{listings}
\usepackage{color}

\definecolor{dkgreen}{rgb}{0,0.6,0}
\definecolor{gray}{rgb}{0.5,0.5,0.5}
\definecolor{mauve}{rgb}{0.58,0,0.82}

\theoremstyle{plain}
\newtheorem{theorem}{Theorem}[section]
\newtheorem*{theorem*}{Theorem}

\newtheorem{lemma}[theorem]{Lemma}
\newtheorem{proposition}[theorem]{Proposition}
\newtheorem{corollary}[theorem]{Corollary}

\newtheorem{claim}[theorem]{Claim}

\theoremstyle{definition}
\newtheorem{definition}[theorem]{Definition}

\theoremstyle{remark}
\newtheorem{remark}[theorem]{Remark}

\numberwithin{equation}{section}

\newcommand{\dd}{\mathrm{d}}

\newcommand{\C}{\mathbb{C}}
\newcommand{\R}{\mathbb{R}}

\newcommand{\Z}{\mathbb{Z}}

\newcommand{\V}{\mathbb{V}}
\newcommand{\X}{\mathbf{X}}

\newcommand{\HH}{\mathbb{H}}

\newcommand{\eps}{\varepsilon}

\newcommand{\mc}{\mathcal}

\DeclareMathOperator{\Tr}{Tr}

\DeclareMathOperator{\E}{\mathcal{E}}

\DeclareMathOperator{\e}{\mathbf{e}}

\DeclareMathOperator{\End}{\mathrm{End}}

\newcommand{\be}{\begin{equation}}
\newcommand{\ee}{\end{equation}}

\def\beq{\begin{equation}}
\def\eeq{\end{equation}}

\def\beq{\begin{equation}}
\def\eeq{\end{equation}}
\def\bea{\begin{eqnarray*}}
\def\eea{\end{eqnarray*}}

\def\nablaV{\nabla_{\V}}

\title
[On the ergodicity of unitary frame flows on K\"ahler manifolds]
{On the ergodicity of unitary frame flows on K\"ahler manifolds}

\author[M. Ceki\'c]{Mihajlo Ceki\'c}
\address{Institut f\"ur Mathematik, Universit\"at Z\"urich, Winterthurerstrasse 190, CH-8057 Z\"urich, Switzerland}
\email{mihajlo.cekic@math.uzh.ch}

\author[T. Lefeuvre]{Thibault Lefeuvre}
\address{Université de Paris and Sorbonne Université, CNRS, IMJ-PRG, F-75006 Paris, France.}
\email{tlefeuvre@imj-prg.fr}

\author[A. Moroianu]{Andrei Moroianu}
\address{Université Paris-Saclay, CNRS,  Laboratoire de mathématiques d'Orsay, 91405, Orsay, France}
\email{andrei.moroianu@math.cnrs.fr}

\author[U. Semmelmann]{Uwe Semmelmann}
\address{Institut f\"ur Geometrie und Topologie, Fachbereich Mathematik, Universit{\"a}t Stuttgart, Pfaffenwaldring 57, 70569 Stuttgart, Germany}
\email{uwe.semmelmann@mathematik.uni-stuttgart.de}

\begin{document}

\begin{abstract} Let $(M,g,J)$ be a closed K\"ahler manifold with negative sectional curvature and complex dimension $m := \dim_{\C} M \geq 2$. In this article, we study the \emph{unitary frame flow}, that is, the restriction of the frame flow to the principal $\mathrm{U}(m)$-bundle $F_{\C}M$ of unitary frames. We show that if $m \geq 6$ is even,  and $m \neq 28$, there exists $\lambda(m) \in (0, 1)$ such that if $(M, g)$ has negative $\lambda(m)$-pinched holomorphic sectional curvature, then the unitary frame flow is ergodic and mixing. The constants $\lambda(m)$ satisfy $\lambda(6) = 0.9330...$, $\lim_{m \to +\infty} \lambda(m) = \tfrac{11}{12} = 0.9166...$, and $m \mapsto \lambda(m)$ is decreasing. This extends to the even-dimensional case the results of Brin-Gromov \cite{Brin-Gromov-80} who proved ergodicity of the unitary frame flow on negatively-curved compact Kähler manifolds of odd complex dimension. 
\end{abstract}

\maketitle

%

\section{Introduction}

\subsection{Ergodicity and mixing of unitary frame flows}

Let $(M,g,J)$ be a smooth closed (compact, without boundary) K\"ahler manifold with \emph{negative sectional curvature} and complex dimension $m \geq 2$. Let $SM \to M$ be the unit tangent bundle and let $F_{\C}M \to M$ be the principal $\mathrm{U}(m)$-bundle of \emph{unitary bases} over $M$. A point $w \in F_{\C}M$ over $x \in M$ is the data of an orthonormal basis $(v, \e_2, ..., \e_{m})$ of $(T_xM, h_x)$, where $h_x(\cdot, \cdot) = g_x(\cdot, \cdot) + ig_x(\cdot, J_x \cdot)$ is the canonical Hermitian inner product on the fibres of $TM$. Equivalently, we will see $F_{\C}M$ as a principal $\mathrm{U}(m - 1)$-bundle over $SM$ by the projection map $p : F_{\C}M \to SM$ defined as $p (v,\e_2,...,\e_{m}) = v$. 

The \emph{geodesic flow} $(\varphi_t)_{t \in \R}$ on $SM$ is defined as $\varphi_t(v) :=\dot{\gamma}_{v}(t)$, where $t \mapsto \gamma_{v}(t) \in M$ is the geodesic generated by $v\in SM$. The \emph{unitary frame flow} on $F_{\C}M$ is then defined as
\begin{equation}
\label{equation:complex-ff}
\Phi_t(v,\e_2,...,\e_m) := (\varphi_t(v), P_{\gamma_{v}(t)}\e_2, ..., P_{\gamma_{v}(t)}\e_m),
\end{equation}
where $P_{\gamma_{v}(t)} : T_x M \to T_{\gamma_{v}(t)}M$ is the parallel transport along $\gamma_{v}$ with respect to the Levi-Civita connection.

Recall that a flow $(\Phi_t)_{t \in \R}$ on a compact metric space $\mc{M}$ is said to be \emph{ergodic} with respect to an invariant probability measure $\mu$ if any flow-invariant function $f \in L^2(\mu)$ is constant. It is said to be \emph{mixing} if for all $f_1, f_2 \in L^2(\mc{M},\mu)$,
\[
\lim_{t\to+\infty}\int_{\mc{M}} f_1\cdot (f_2 \circ \Phi_t)\, \dd \mu = \int_\mc{M} f_1\, \dd \mu \cdot\int_\mc{M} f_2\, \dd \mu.
\]

While the geodesic flow $(\varphi_t)_{t \in \R}$ of any negatively curved compact Riemannian manifold is well-known to be ergodic \cite{Hopf-36,Anosov-67} with respect to the Liouville measure on $SM$, on negatively curved compact Kähler manifolds, the ergodicity of the unitary frame flow $(\Phi_t)_{t \in \R}$ with respect to the natural flow-invariant smooth measure $\omega$ induced by the Liouville measure and the Haar measure on the group $\mathrm{U}(m-1)$, is a much more difficult question due to its lack of uniform hyperbolicity. It was proved by Brin-Gromov \cite{Brin-Gromov-80} that this flow is ergodic whenever $m:=\dim_{\C} M$ is odd or $m=2$ but the even-dimensional case $m\ge 4$ has remained open so far. The aim of this paper is to bring a first positive answer when $m\ge 6$ is even and $m \neq 28$, under some pinching hypothesis for the sectional curvature.

Recall that the \emph{holomorphic sectional curvature} of $(M,g,J)$ is defined as
\begin{equation}
\label{equation:holomorphic-intro}
H(X) := R(X,JX,JX,X), 
\end{equation}
for all unit vectors $X \in TM$, where $R$ is the Riemann curvature tensor of $(M,g)$. The manifold is said to be \emph{holomorphically $\lambda$-pinched},  for some $\lambda \in (0, 1]$, if there exists a constant $C > 0$ such that
\begin{equation}
\label{equation:pinching-intro}
-C \leq H \leq -C \lambda.
\end{equation}
The manifold is said to be \emph{strictly} holomorphically $\lambda$-pinched if the inequalities in \eqref{equation:pinching-intro} are strict.

In order to state our main result, we introduce the function $m \mapsto \lambda(m)$, defined for even numbers $m \geq 6$ by
\begin{equation}
\label{equation:plot}
\lambda(m): =\dfrac{308m+131}{336m+105}.
\end{equation}
The function $m \mapsto \lambda(m)$ is decreasing, $\lambda(6) = 0.9330...$ and $\lim_{m \to +\infty} \lambda(m) = \tfrac{11}{12} = 0.9166...$.

We will prove that the following holds:

\begin{theorem}
\label{theorem:main}
Let $(M,g,J)$ be a closed connected K\"ahler manifold of complex dimension $m \geq 2$, with negative sectional curvature. The unitary frame flow $(\Phi_t)_{t \in \R}$ on $F_{\C}M$ is \emph{ergodic} and \emph{mixing} with respect to the smooth measure $\omega$ if:
\begin{enumerate}[label=\emph{(\roman*)}]
\item The complex dimension $m$ is odd or $m=2$ \cite{Brin-Gromov-80},
\item The complex dimension $m\ge 6$ is even, $m \neq 28$, and the manifold is strictly holomorphically $\lambda(m)$-pinched. 
\end{enumerate}
\end{theorem}


We will actually show that the unitary frame flow is ergodic if and only if it is mixing. We believe that ergodicity should hold \emph{without any pinching condition} but it is clear from the proofs that our method only works with a pinching condition close to $1$.

In the case of constant holomorphic curvature $H = -1$ (that is, on compact quotients $\Gamma \backslash \C\HH^m$ of the complex hyperbolic space), the ergodicity of unitary frame flow was shown by Howe-Moore \cite{Howe-Moore-79}. In variable holomorphic curvature, besides \cite{Brin-Gromov-80} in odd complex dimensions and $m=2$, Theorem \ref{theorem:main} seems to be the first result proving ergodicity of unitary frame flows on negatively-curved Kähler manifolds of even complex dimensions $m\ge 6$.  

As indicated in Theorem \ref{theorem:main}, it also seems that our technique does not apply in complex dimensions $m=4$ and $m=28$. The former case is related to the fact that $S^7$ is parallelizable, whereas the latter case is connected to an open problem in algebraic topology which is to classify reductions of the structure group of the unitary frame bundle $F_{\C}S^{55}$ over the sphere $S^{55}$. More precisely, we are unable to rule out the possible existence of a special $\mathrm{E}_6$-structure on $S^{55}$ and this eventually turns out to be problematic in order to run our argument, see \S\ref{section:topology} where this is further discussed.

The structure of the argument, described with more details in \S\ref{ssection:structure}, is somewhat similar to our previous article \cite{Cekic-Lefeuvre-Moroianu-Semmelmann-21} proving ergodicity of \emph{real frame flows}\footnote{In the literature, the word ``frame flow'' usually refers to what we call here the ``real frame flow'', that is, the parallel transport of all bases regardless of any almost-complex structure. We added the word ``unitary'' in the K\"ahler case and ``real'' in the Riemannian case in order to make a distinction.} on negatively-curved compact Riemannian manifolds of even real dimensions with nearly $0.25$-pinched (real) sectional curvature, thus almost answering a long-standing conjecture of Brin, see \cite[Conjecture 2.9]{Brin-82}. Nevertheless, the present article is not a mere adaptation of \cite{Cekic-Lefeuvre-Moroianu-Semmelmann-21} as we had to develop new techniques in order to take into account the specificities of the Kähler setting, see Theorem \ref{theorem:topology}, \S\ref{ssection:cntckt} or \S\ref{section:5.3} for instance. Although it is meant to be self-contained, we encourage the reader to consult \cite{Lefeuvre-21,Cekic-Lefeuvre-Moroianu-Semmelmann-21, Cekic-Lefeuvre-Moroianu-Semmelmann-22} as we build here on the framework developed in these articles.

Prior to \cite{Cekic-Lefeuvre-Moroianu-Semmelmann-21}, the real frame flow was known to be ergodic on odd-dimensional negatively-curved Riemannian manifolds (of dimension $\neq 7$) by Brin-Gromov \cite{Brin-75-1,Brin-Gromov-80} and in even dimensions (and dimension $7$) for manifolds with a pinching close to $1$ by Brin-Karcher \cite{Brin-Karcher-84} and Burns-Pollicott \cite{Burns-Pollicott-03}.

The real frame flows are historical examples of \emph{partially hyperbolic flows} studied in the aftermath of Anosov's seminal work on hyperbolic flows \cite{Anosov-67} by Brin and Pesin \cite{Brin-Pesin-74, Brin-75-1,Brin-75-2}. The field of partially hyperbolic dynamical systems is now a well-established and active field of dynamical systems, see \cite{Hasselblatt-Pesin-06} for instance for an introduction to this topic.

Eventually, let us mention that, similarly to the real case where ergodicity of the real frame flow was shown to determine the high-energy behaviour of eigenfunctions of Dirac-type operators \cite{Jakobson-Strohmaier-07}, the ergodicity of the unitary frame flow on Kähler manifolds determines the high-energy behaviour of eigenfunctions of Dolbeault Laplacians and Spin${}^c$ Dirac operators \cite{Jakobson-Strohmaier-Zelditch-08}.

\subsection{Proof ideas} \label{ssection:structure}


Let us summarise the argument which, as mentioned above, is similar to the one developed in \cite{Cekic-Lefeuvre-Moroianu-Semmelmann-21} and consists of three steps:

\begin{enumerate}[label=(\roman*), itemsep=4pt]
\item \textbf{Hyperbolic dynamics:} Following Brin's ideas \cite{Brin-75-1} (see also \cite{Lefeuvre-21} for a more recent approach), the non-ergodicity of the unitary frame flow is described by means of a subgroup $H \lneqq \mathrm{U}(m-1)$ called the \emph{transitivity group}, see \S\ref{ssection:isometry}. In particular, there exists a smooth flow-invariant principal $H$-subbundle $Q \subset F_{\C} M$ over $SM$, such that the restriction of $(\Phi_t)_{t \in \R}$ to $Q$ is ergodic.

\item \textbf{Algebraic topology:} The group $H$ thus provides a reduction of the structure group of $F_{\C}M$ from $\mathrm{U}(m-1)$ to $H$. In particular, restricting to a point $x_0 \in M$ and identifying $S_{x_0}M \simeq S^{2m-1}$, the unitary frame bundle $F_{\C}S^{2m-1} \to S^{2m-1}$ must admit a reduction of its structure group to $H$. In \S\ref{section:topology}, we classify such reductions and show that, for $m \neq 4,28$, $H$ must act reducibly on $\C^{m-1}$.

\item \textbf{Riemannian geometry:} Using the non-Abelian Liv\v sic Theorem of \cite{Cekic-Lefeuvre-21-1}, we then deduce that there exists a smooth flow-invariant complex vector bundle $\mc{V} \to SM$ which is a subbundle $\mc{V} \subset \pi^*TM$ (where $\pi : SM \to M$ is the projection) satisfying certain algebraic properties. In turn, using the twisted Pestov identity (see Lemma \ref{lemma:pestov}), we rule out the existence of such an object under a certain pinching condition $\lambda > \lambda(m)$ in \S\ref{section:pestov} and \S\ref{section:threshold}.
\end{enumerate}


\subsection{Structure of the article} In \S\ref{section:preliminaries}, we recall standard facts from Riemannian and complex geometry, and (partially) hyperbolic dynamical systems needed in the rest of the article. In \S\ref{section:topology}, we study the possible reductions of the structure group of the unitary frame bundle over the sphere, and deduce the existence of non-zero flow-invariant projectors whenever the frame flow is not ergodic. In \S\ref{section:pestov}, we derive, using the twisted Pestov identity, an inequality that must be satisfied by such an invariant object. In \S\ref{section:threshold}, we complete the proof of Theorem \ref{theorem:main}.

\subsection*{Acknowledgements} We thank Maxime Zavidovique, one of the participants of the \emph{Geometry and Topology seminar} in Jussieu, for pointing out that this problem could be worth studying (although we first thought that it had already been solved a long time ago)! M.C. has received funding from an Ambizione grant (project number
201806) from the Swiss National Science Foundation.

\section{Preliminaries}

\label{section:preliminaries}

\subsection{Riemannian geometry of the unit tangent bundle} Let $(M,g)$ be a closed connected Riemannian manifold of real dimension $n$. Denote by
\[
SM := \left\{ v \in TM ~|~ |v|_g =1\right\}
\]
the unit tangent bundle of $(M,g)$ and by $\pi : SM \to M$ the projection map.

\subsubsection{Tangent bundle of $SM$}

 Let $(\varphi_t)_{t \in \R}$ be the geodesic flow on $SM$ with generating vector field $X \in C^\infty(SM,T(SM))$. The tangent bundle $T(SM)$ splits as
\begin{equation}
\label{equation:splitting-tsm}
T(SM) = \V \oplus \HH \oplus \R X,
\end{equation}
where $\V := \ker \dd\pi$ is the \emph{vertical bundle}, and $\HH$ is the \emph{horizontal bundle} defined by means of the Levi-Civita connection, see \cite[Chapter 1]{Paternain-99}. The metric $g$ induces a canonical metric on $T(SM)$ called the \emph{Sasaki metric} such that the splitting \eqref{equation:splitting-tsm} is orthogonal.

If $f \in C^\infty(SM)$ is a smooth function, its gradient $\nabla f \in C^\infty(SM,T(SM))$ computed with respect to the Sasaki metric splits according to \eqref{equation:splitting-tsm} as
\[
\nabla f = \nabla_\V f + \nabla_{\HH} f + (Xf) X,
\]
where $\nabla_\V f \in C^\infty(SM, \V)$ is the \emph{vertical gradient} and $\nabla_{\HH}f \in C^\infty(SM,\HH)$ is the \emph{horizontal gradient}. The $L^2$-norm on $SM$ is defined using the \emph{Liouville measure} $\mu$ on $SM$ which is the Riemannian measure induced by the Sasaki metric.
Note that the Liouville measure is \emph{invariant} by the geodesic flow.

The \emph{vertical Laplacian} $\Delta_{\V}$ is then defined as $\Delta_{\V} := \nabla_{\V}^* \nabla_{\V}$, where $\nabla_{\V}^*$ denotes the $L^2$-adjoint. Equivalently, given $f \in C^\infty(SM)$ and $x \in M$, denoting the Laplacian of the restriction of $g_x$ to $S_xM$ by $\Delta_{S_xM}$, we have
\begin{equation}
\label{equation:deltav}
\Delta_{\V}f(v) = \Delta_{S_xM}(f|_{S_xM})(v),\qquad\forall v\in S_xM.
\end{equation}

\subsubsection{Fourier decomposition in the fibers} Since $\pi : SM \to M$ is a sphere bundle, any smooth function $f \in C^\infty(SM)$ can be decomposed into a sum of spherical harmonics on the sphere $S_xM \simeq S^{n-1}$ above each point $x \in M$. In other words, we can write
\[
f = \sum_{k=0}^{+\infty} f_k,
\]
where $f_k \in C^\infty(SM)$ is a spherical harmonic of degree $k \geq 0$, that is, it satisfies the eigenvalue equation
\[
\Delta_{\V} f_k = k(n+k-2) f_k,
\]
where $\Delta_{\V}$ is the vertical Laplacian on each sphere $S_xM$ (for $x \in M$) introduced in \eqref{equation:deltav}. The space of spherical harmonics of degree $k$ over $M$ defines a vector bundle $\Omega_k \to M$ which can be naturally identified with the vector bundle of trace-free symmetric $k$-tensors $S^k_0 TM \to M$ via the map  (here we identify $T^*M$ and $TM$ by using the metric $g$)
\begin{equation}
\label{equation:iso}
\pi_k^* : S^k_0 TM \overset{\sim}{\longrightarrow} \Omega_k, \qquad \pi_k^* f (v) := f_x(v,...,v),\qquad\forall v\in S_xM.
\end{equation}

More generally, let $(E, h) \to M$ be a Hermitian (or Euclidean) vector bundle over $M$ equipped with a unitary (or orthogonal) connection $\nabla^E$,  by which we mean that
\[Y h(e, f) = h(\nabla^E_Ye, f) + h(e, \nabla^E_Y f), \quad \forall e, f \in C^\infty(M, E), \quad \forall Y \in C^\infty(M, TM).\]
Denoting by $(\mc{E},\nabla^{\E}) := (\pi^* E,\pi^*\nabla^{E})$ its pullback to $SM$, any section $f \in C^\infty(SM,\mc{E})$ can be uniquely decomposed into a sum of twisted spherical harmonics over each point $x \in M$, that is,
\begin{equation}
\label{equation:sum}
f = \sum_{k=0}^{+\infty} f_k,
\end{equation}
where $f_k \in C^\infty(SM,\mc{E})$ is a spherical harmonic of degree $k$ (with values in $\mc{E}$). Note that, with respect to an orthonormal basis $(\e_\alpha)$ on $E$ defined locally over $U \subset M$, any section $f \in C^\infty(SM,\mc{E})$ can be written as $f|_{U} = \sum_{\alpha} f_\alpha \e_\alpha$, where $f_\alpha \in C^\infty(SM|_U)$. Then, the vertical Laplacian is defined as
\[
\Delta_{\V}^E f = \sum_\alpha (\Delta_{\V} f_\alpha) \e_\alpha,
\]
where $\Delta_{\V}$ was introduced in \eqref{equation:deltav}. The sections $f_k \in C^\infty(SM,\mc{E})$ then satisfy the eigenvalue equation
\[
\Delta_{\V}^E f_k = k(n+k-2) f_k.
\]
Equivalently, $f_k$ is a smooth section of the bundle $\Omega_k \otimes E$ over $M$ and this can be identified via \eqref{equation:iso} to an element $S^k_0 TM \otimes E$. We say that a section $f \in C^\infty(SM,\mc{E})$ has \emph{even} (resp. \emph{odd}) Fourier degree, if the decomposition \eqref{equation:sum} only involves spherical harmonics of even (resp. odd) degree.

We define the operator $\X := \nabla^{\E}_X$, where $X$ is the geodesic vector field on $SM$. This is the infinitesimal generator of the parallel transport of sections of $E$ along geodesic flow-lines. It has the mapping property
\begin{equation}
\label{equation:split}
\X : C^\infty(M,\Omega_k \otimes E) \to C^\infty(M,\Omega_{k-1} \otimes E) \oplus C^\infty(M,\Omega_{k+1} \otimes E),
\end{equation}
and therefore splits as a sum $\X := \X_- + \X_+$, where each term corresponds to the two summands in \eqref{equation:split}.

There is a natural $L^2$ scalar product on sections $f,f' \in C^\infty(SM, \E)$ given by:
\begin{equation}
\label{equation:l2}
\langle f,f'\rangle_{L^2} :=\int_{SM}h_{\pi(v)}(f(v),f'(v))\,\dd\mu,
\end{equation}
where $\mu$ is the Liouville measure on $SM$, and $h$ is the Hermitian (or Euclidean) metric on $E$.

\subsubsection{Twisted Pestov identity} This identity will play a fundamental role in our proof of Theorem \ref{theorem:main}. In the non-twisted case, it was first discovered in some particular cases by Mukhometov \cite{Mukhometov-75,Mukhometov-81} and Amirov \cite{Amirov-86}, and then in its classical shape by Pestov and Sharafutdinov \cite{Pestov-Sharafutdinov-88, Sharafutdinov-94}. Eventually, it was reformulated and described in a general coordinate-free way in \cite{Paternain-Salo-Uhlmann-15, Guillarmou-Paternain-Salo-Uhlmann-16}. It takes the following form: 

\begin{lemma}[Localized Pestov identity]
\label{lemma:pestov}
Let $(M,g)$ be a closed $n$-dimensional Riemannian manifold and let $(E,h)$ be a Hermitian (or Euclidean) vector bundle over $M$ equipped with unitary (or orthogonal) connection $\nabla^E$. The following identity holds for all $k \in \Z_{\geq 0}$ and $u \in C^\infty(M,\Omega_k \otimes E)$:
\begin{equation}
\label{equation:pestov}
\begin{split}
\tfrac{(n+k-2)(n+2k-4)}{n+k-3} \|\X_-u\|^2_{L^2}& -   \tfrac{k(n+2k)}{k+1} \|\X_+u\|^2_{L^2}  + \|Z(u)\|^2_{L^2}  \\
&= \langle R\nabla_{\V}^{E}u, \nabla_{\V}^{E}u \rangle_{L^2} + \langle \mc{F}^{E}u, \nabla_{\V}^{E}u \rangle_{L^2},
\end{split}
\end{equation}
where:
\begin{itemize}
\item $Z$ is a first order differential operator which we do not make explicit,
\item the term involving $R$ takes the form
\[
\langle R\nabla_{\V}^{E}u, \nabla_{\V}^{E}u \rangle_{L^2} =  \int_{M} \int_{S_xM} \sum_\alpha R(v, \nabla_{\V}u_{\alpha},\nabla_{\V}u_{\alpha},v) ~|\dd v| |\dd x|,
\]
where $R$ is the Riemann curvature tensor, $|dv|$ is the Lebesgue measure on the sphere $S_xM$ (induced by $g_x$) and $|dx|$ is the Riemannian measure, $u = \sum_\alpha u_\alpha \e_\alpha$ with $(\e_\alpha)_{\alpha \in I}$ an orthonormal frame at the point $x \in M$ of $E_x$,
\item the term involving $\mc{F}^{E}$ takes the form
\begin{equation}
\label{equation:f}
\langle \mc{F}^{E}u, \nabla_{\V}^{E}u \rangle_{L^2} = \int_M \int_{S_xM}  \sum_{\alpha} R_E(v,\nabla_{\V} u_\alpha,u,\e_\alpha) ~ |\dd v| |\dd x|,
\end{equation}
where $R_E$ is the curvature tensor of $E$ and we use the convention:
\[
R_E(X,Y,\omega,\eta) := h(R_E(X,Y)\omega,\eta), \quad \forall X,Y \in TM,\ \forall \omega, \eta \in E,
\]
and $h$ is the Euclidean (or Hermitian) metric on the bundle $E$.
\end{itemize}
\end{lemma}

We refer to \cite[Proposition 3.5]{Guillarmou-Paternain-Salo-Uhlmann-16} for a proof.

\subsection{Complex geometry} We use \cite[Chapter IX]{Kobayashi-Nomizu-96} as basic reference for complex geometry.

\subsubsection{Curvature tensors}\label{sssc:curvature-tensors} Let $(V, g)$ be a Euclidean vector space of dimension $n$. We will usually identify $V$ with its dual $V^*$ and $\Lambda^2 V $ with the space of skew-symmetric endomorphisms of $V$ using the metric $g$. We denote by $S^p V$ the symmetric $p$-tensors on $V$, $S^p_0 V$ the trace-free symmetric tensors and $\Lambda^pV$ the $p$-th exterior power. The space $V^{\otimes 2}$ splits as
\begin{equation}
\label{equation:v2}
V^{\otimes 2} = \R g \oplus S^2_0 V \oplus \Lambda^2 V,
\end{equation}
where each summand is invariant under the $\mathrm{O}(n)$-action, and $S^2 V = \R g \oplus S^2_0 V$. The space $V^{\otimes 2} \simeq \End V$ is equipped with the norm
\begin{equation}
\label{equation:trace}
\langle u, v \rangle = \Tr(u^\top v),
\end{equation}
where ${}^\top$ denotes the transpose operator and \eqref{equation:v2} is orthogonal with respect to \eqref{equation:trace} so that $S^2 V$ and $\Lambda^2 V$ both inherit the metric \eqref{equation:trace}.

A \emph{curvature tensor} $R$ is an element $R \in S^2(\Lambda^2 V)$ satisfying the Bianchi identity
\begin{equation}
\label{equation:bianchi}
R(X,Y,Z,W) + R(Z,X,Y,W)+ R(Y,Z,X,W) = 0, \quad \forall X,Y,Z,W \in V.
\end{equation}
The \emph{sectional curvature} associated to $R$ is the quadratic map $\overline R:S^2V\to \R$ defined by
\[
\overline{R}(X,Y) := R(X,Y,Y,X), \quad X, Y \in V.
\]
Given $X,Y \in V$, we can see $R(X,Y,\cdot,\cdot)$ as a skew-symmetric endomorphism $R(X,Y)$ as follows:
\begin{equation}
\label{equation:equivalent}
\langle R(X,Y)Z,W \rangle:= R(X,Y,Z,W).
\end{equation}
This skew-symmetric endomorphism extends as a derivation to skew-symmetric endomorphisms of the exterior, symmetric and tensor algebras of $V$, denoted respectively by  $R_{\Lambda^p}(X,Y)$, $R_{S^p}(X,Y)$ and $R_{V^{\otimes p}}(X,Y)$.
In particular, it can be easily checked that:
\begin{equation}
\label{equation:curv-induced}
R_{V^{\otimes 2}}(X,Y) u = [R(X,Y), u]
\end{equation}
for every $u \in V^{\otimes 2} =\End(V)$.
The action \eqref{equation:curv-induced} is diagonal with respect to the decomposition $V^{\otimes 2} = S^2 V \oplus \Lambda^2 V$.  For $X,Y \in V$ and $\omega, \eta \in \Lambda^p V$, we set
\[
R_{\Lambda^p V}(X,Y,\omega,\eta) := \langle R_{\Lambda^p V}(X,Y)\omega, \eta \rangle,
\]
where $\Lambda^p V$ is equipped with the canonical inner product. We use the analogous notation for $S^pV$.

\subsubsection{Curvature and pinching} If $(M,g)$ is a Riemannian manifold, we introduce the $(4,0)$-tensor $g\owedge g$ by:
\begin{equation}
\label{equation:curvature-h}
g \owedge g (X,Y,Z,W) := g( X,Z)g( Y,W) - g( X,W)g( Y,Z), 
\end{equation}
for all $X,Y,Z,W \in TM.$
This is precisely the curvature tensor when $(M,g)$ is the real hyperbolic space $\HH^n$, whereas if $(M,g,J)$ is the complex hyperbolic space $\C \HH^m$, its curvature tensor $G$ takes the form:
\begin{equation}
\label{eq:complexhyp}
\begin{split}
4G(X, Y, Z, W) = &g\owedge  g(X, Y, Z, W)  + g\owedge g(X, Y, JZ,JW)  \\
& \qquad +2g(X, JY)g(Z, JW),
\end{split}
\end{equation}
see \cite[Section 7, Chapter IX]{Kobayashi-Nomizu-96}. Equivalently, \eqref{eq:complexhyp} can be rewritten using \eqref{equation:equivalent} as
\begin{equation}
\label{equation:G}
4G(X,Y) = X \wedge Y + JX \wedge JY - 2 \langle X,JY\rangle J,
\end{equation}
where $X \wedge Y$ is the skew-symmetric endomorphism of $TM$ defined by $(X\wedge Y)(Z):=g(X,Z)Y-g(Y,Z)X$ for all $Z\in TM$.

If $(M,g,J)$ is any Kähler manifold, the \emph{holomorphic sectional curvature} is defined for a unit vector $X \in TM$ by:
\begin{equation}
\label{equation:holomorphic}
H(X) := \overline{R}(X,JX) = R(X,JX,JX,X).
\end{equation}
It can be easily checked that the holomorphic curvature of the complex hyperbolic space is $-1$, that is,
\[
H_{\C\HH^m}(X) = G(X,JX,JX,X) = -1,
\]
for $|X|=1$. By analogy with the real case, we introduce the notion of pinching of the holomorphic curvature:

\begin{definition}[Pinched holomorphic sectional curvature]
We say that a Kähler manifold $(M^{2m},g,J)$ is negatively \emph{holomorphically $\lambda$-pinched} (for some $\lambda \in (0,1]$) if there exists a constant $C > 0$ such that for all unit vectors $X \in TM$, we have
\begin{equation}
\label{equation:holomorphic-pinching}
-C \leq H(X) \leq -C \lambda.
\end{equation}
The manifold is said to be \emph{strictly} negatively holomorphically $\lambda$-pinched if the above inequalities are strict.
\end{definition}

Similarly, one can talk about the \emph{real} (or \emph{sectional}) $\delta$-pinching of the manifold $(M^{2m},g)$ by requiring that \eqref{equation:holomorphic-pinching} holds with $\lambda$ being replaced by $\delta$, and $H(X)$ being replaced by the sectional curvature $\overline{R}(X,Y)$ (for all pairs of orthogonal unit vectors $X,Y \in TM$). There exist relations between holomorphic and real pinchings, see \cite{Berger-60-1,Berger-60-2, Bishop-Goldberg-63}.

As in the real case, it is a well known result that negative holomorphic $1$-pinching implies that $(M^{2m},g,J)$ is holomorphically isometric to a compact quotient $\Gamma \backslash \C\HH^m$, where $\Gamma$ is a discrete subgroup of $\mathrm{Isom}(\C\HH^m)$. In what follows, we will always assume that $(M^{2m},g,J)$ is negatively $\lambda$-holomorphically pinched and, without loss of generality, we rescale the metric such that $C=1$.

The following lemma proved in \cite[Proposition 4.2]{Bishop-Goldberg-63} will be useful:

\begin{lemma}[Bishop-Goldberg '63]
\label{lemma:bishop-goldberg}
Assume $(M^{2m},g,J)$ is a closed Kähler manifold which is negatively holomorphically $\lambda$-pinched. Consider unit vectors $X,Y \in TM$ such that $g( X,Y) = 0$ and $g( X, JY) = \cos \theta$. Then:
\begin{equation}
\label{equation:sect-pinching0}
-( 1- \tfrac{3}{4} \lambda \sin^2\theta) \leq \overline{R}(X,Y) \leq - \tfrac{1}{4}\left( 3(1+\cos^2\theta)\lambda-2\right).
\end{equation}
In particular:
\begin{equation}
\label{equation:sect-pinching}
-1 \leq \overline{R}(X,Y) \leq - \tfrac{3\lambda-2}{4},
\end{equation}
for all orthogonal unit vectors $X, Y \in TM$. If $\lambda \geq 2/3$, then $(M^{2m},g)$ is negatively $\delta$-pinched for the sectional curvature with $\delta = \tfrac{3\lambda-2}{4}$.
\end{lemma}

We now set
\begin{equation}
\label{equation:decomp-r}
R = R_0 + \tfrac{1+\lambda}{2}G,
\end{equation}
where $G$ is the  curvature tensor defined in \eqref{eq:complexhyp}. The following holds:

\begin{lemma}
\label{lemma:r0}
Assume that $-1 \leq H(X) \leq -\lambda$ for all unit vectors $X\in TM$. Then, for all unit vectors $X,Y,Z,W \in TM$, one has:
\begin{equation}
\label{equation:bound-r0bar}
|R_0(X,Y,Z,W)| \leq \tfrac{4}{3}(1-\lambda).
\end{equation}
More generally, for all unit vectors $X,Y \in TM$, and for all unit  $\omega,\eta \in \Lambda^p TM$ or $S^p TM$, one has:
\begin{equation}
\label{equation:bound-r0bar-general}
|(R_0)_{\Lambda^p TM}(X,Y, \omega, \eta)|, |(R_0)_{S^p TM}(X,Y, \omega, \eta)| \leq \tfrac{4p}{3}(1-\lambda).
\end{equation}
\end{lemma}

\begin{proof}
Let $X,Y$ be unit vectors such that $g( X,Y) = 0$ and $g( X, JY) = \cos \theta$. We have:
\[
\overline{R_0}(X,Y) = \overline{R}(X,Y)-\tfrac{1+\lambda}{2}\overline{G}(X,Y) = \overline{R}(X,Y)+\tfrac{1+\lambda}{2}(1-\tfrac{3}{4}\sin^2\theta).
\]
Inserting \eqref{equation:sect-pinching0} in the previous equation, we get:
\[
-\tfrac{1-\lambda}{2}(1+\tfrac{3}{4} \sin^2\theta) \leq \overline{R_0}(X,Y) \leq \tfrac{1-\lambda}{2}(2 - \tfrac{3}{4}\sin^2\theta).
\]
In particular, the previous inequalities yield:
\begin{equation}
\label{equation:r0bar}
|\overline{R_0}(X,Y)| \leq 1-\lambda.
\end{equation}
Like in the proof of \cite[Lemma 3.7]{Bourguignon-Karcher-78}, \eqref{equation:r0bar} then implies \eqref{equation:bound-r0bar} by writing $R_0(X,Y,Z,W)$ as a sum of terms only involving two vectors in the arguments. The general bound \eqref{equation:bound-r0bar-general} follows immediately from \eqref{equation:bound-r0bar} by diagonalizing over $\C$ the skew-symmetric endomorphism $R_0(X,Y)$.
\end{proof}

\subsection{Isometric extensions of the geodesic flow} \label{ssection:isometry} The unitary frame bundle $\widehat{\pi} : F_{\C}M \to SM$ is a principal $\mathrm{U}(m-1)$-bundle over $SM$. Given $a \in \mathrm{U}(m-1)$, we denote by $R_a : F_{\C}M \to F_{\C}M$ the fiberwise right-action by $a$. The unitary frame flow $(\Phi_t)_{t \in \R}$ is an \emph{extension} of the geodesic flow to a principal bundle in the sense that it satisfies
\[
\pi \circ \Phi_t = \varphi_t \circ \pi, \qquad R_a \circ \Phi_t = \Phi_t \circ R_a,
\]
for all $t \in \R, a \in \mathrm{U}(m-1)$. We will denote by $X_{F_{\C}M}$ its infinitesimal generator.

Initiated by the work of Brin \cite{Brin-75-1,Brin-75-2}, there is now an established theory describing the ergodic components of such an extension flow $(\Phi_t)_{t \in \R}$. This is achieved via the introduction of a closed subgroup $H \leqslant \mathrm{U}(m-1)$, called the \emph{transitivity group}, and defined by means of dynamical holonomies. We refer to \cite{Lefeuvre-21} for a modern construction of the transitivity group $H$. It has the following properties:

\begin{theorem}[Brin '75, \cite{Lefeuvre-21}]
\label{theorem:brin-lefeuvre}
The following holds:
\begin{enumerate}[label=\emph{(\roman*)}]
\item There exists a natural isomorphism
\begin{equation}
\label{equation:ev}
\mathrm{ev} : \ker_{L^2(\omega)} X_{F_{\C}M} \overset{\sim}{\longrightarrow} L^2(H\backslash \mathrm{U}(m-1)).
\end{equation}
\item There exists a principal $H$-subbundle $Q \subset F_{\C}M$ over $SM$ which is invariant by $(\Phi_t)_{t \in \R}$ such that $(\Phi_t|_{Q})_{t \in \R}$ is ergodic (with respect to the induced measure on $Q$).
\end{enumerate}
In particular, the unitary frame flow is ergodic if and only if $H = \mathrm{U}(m-1)$.
\end{theorem}

We refer to \cite[Corollary 3.10]{Lefeuvre-21} for further details. Obviously, by the first item, the only possibility for $\ker_{L^2} X_{F_{\C}M}$ to be reduced to the constants is that $H = \mathrm{U}(m-1)$. The isomorphism in \eqref{equation:ev} is simply defined by taking an arbitrary point $z_\star \in SM$ and setting for $f \in \ker_{L^2} X_{F_{\C}M}$,
\[
\mathrm{ev}(f) := f|_{F_{\C}M_{z_\star}}.
\]
Such a function turns out to be in $L^2(F_{\C}M_{z_\star}) \simeq L^2(\mathrm{U}(m-1))$\footnote{It is not clear \emph{a priori} that such an evaluation map is well-defined, so its definition is also part of Theorem \ref{theorem:brin-lefeuvre}.} and is invariant by the action of $H$ so it yields an element in $L^2(H\backslash \mathrm{U}(m-1))$. The second item in Theorem \ref{theorem:brin-lefeuvre} is already a strong topological constraint on the bundle $F_{\C}M$ and is called a \emph{reduction of the structure group}, see \S\ref{section:topology} where this is further discussed. \\

We introduce $\mc{N} \to SM$, the \emph{normal bundle}, to be the Euclidean bundle over $SM$ defined for $v\in S_xM$ as:
\[
\mc{N}(v) := \mathrm{Span}(v,Jv)^\perp,
\]
where $\perp$ denotes the orthogonal complement with respect to the Euclidean metric $g_x$. Note that $\mc{N}$ is equipped with the complex structure $J$. Observe that
\begin{equation}
\label{equation:n}
\pi^* TM = \mc{N} \oplus \R v \oplus \R Jv,
\end{equation}
so that $\mc{N}$ can be seen as a subbundle of the pullback bundle $\pi^*TM$. Parallel transport with respect to the Levi-Civita connection $\nabla^{\mathrm{LC}}$ of sections of $\mc{N}$ along geodesic flow-lines is well-defined and generated by a first order differential operator
\[
\X:= (\pi^*\nabla^{\mathrm{LC}})_X : C^\infty(SM,\mc{N}) \to C^\infty(SM,\mc{N}),
\]
which is formally skew-adjoint and commutes with $J$.

Other than describing the ergodic components of the unitary frame flow, the group $H$ allows to construct smooth flow-invariant objects. In what follows, we denote by $\mathrm{Vect}$ the category of finite-dimensional Euclidean vector spaces and call $\mathfrak{o} : \mathrm{Vect} \to \mathrm{Vect}$ an \emph{operation} on this category if $\mathfrak{o}$ is obtained as a finite composition of the following basic operations: tensor powers $V^{\otimes m}$ of a vector space $V$, symmetric $S^m V$ and exterior powers $\Lambda^m V$. Obviously, for any such operation $\mathfrak{o}$, $\mathfrak{o}(\mc{N}) \to SM$ is a well-defined Euclidean bundle still equipped with an induced generator $\X$ (for simplicity, we do not introduce any new notation for the generator on this bundle).

The following holds:

\begin{theorem}{\em (Non-Abelian Liv\v sic Theorem, \cite[Theorem 3.5]{Cekic-Lefeuvre-21-1}).}
\label{theorem:non-abelian-livsic}
Let
\[
\mathfrak{o} : \mathrm{Vect} \to \mathrm{Vect},
\]
be any operation on $\mathrm{Vect}$. Then, there exists an isomorphism
\[
\ker \X \cap C^\infty(SM, \mathfrak{o}(\mc{N})) \overset{\sim}{\longrightarrow} \left\{ f \in \mathfrak{o}(\R^{2(m-1)}) ~|~ hf = f, \forall h \in H\right\}.
\]
\end{theorem}

The isomorphism map is nothing but evaluation at an arbitrary point of $SM$ (similarly to Theorem \ref{theorem:brin-lefeuvre}, (i)). In other words, flow-invariant smooth sections of tensor products of the normal bundle correspond exactly to algebraic $H$-invariant objects on $\R^{2(m-1)}$. Theorem \ref{theorem:non-abelian-livsic} will allow us to generate smooth flow-invariant sections when the unitary frame flow is not ergodic. For instance, if one can show that $H \leqslant \mathrm{U}(m-1-p) \times \mathrm{U}(p) \lneqq \mathrm{U}(m-1)$, that is, $H$ acts reducibly on $\R^{2(m-1)} \simeq \C^{m-1}$, then $H$ fixes an orthogonal projector $\pi \in S^2 \R^{2(m-1)}$ onto a complex (i.e. $J$-invariant) space $V \subset \R^{2(m-1)}$. In turn, Theorem \ref{theorem:non-abelian-livsic} implies that there exists a flow-invariant complex vector bundle $\mc{V} \subset \mc{N}$ which is the same as the existence of an orthogonal projector $\pi_{\mc{V}} \in C^\infty(SM, S^2 \mc{N}) \cap \ker \X$ commuting with $J$.

\section{Topological reductions and flow-invariant sections}

\label{section:topology}

In what follows, we assume that the complex dimension of $M$ is even and larger than $2$, and we write it as $\dim_{\C} M = m = :2p+2$, with $p \geq 1$.

\subsection{Topological reductions}

By Theorem \ref{theorem:brin-lefeuvre}, if the unitary frame flow on $F_{\C}M$ is not ergodic, its transitivity group is a strict subgroup $H \lneqq \mathrm{U}(2p+1)$ and there exists a strict principal $H$-subbundle $Q \subset F_{\C}M$. This is known as a \emph{reduction of the structure group} of $F_{\C}M$ to $H$. Since $F_{\C}M$ admits a reduction to $H$, the same holds true for the restriction of the unitary frame bundle to any sphere $S_{x_0}M$, for $x_0 \in M$. In turn, as $S_{x_0}M \simeq S^{4p+3}$, this implies that the unitary frame bundle $F_\C S^{4p+3}$, admits a reduction of its structure group from $\mathrm{U}(2p+1)$ to $H$. 

Note that $\mathrm{U}(2p+2)$  and $\mathrm{SU}(2p+2)$ act transitively on $S^{4p+3}$ with isotropy groups $\mathrm{U}(2p+1)$ and $\mathrm{SU}(2p+1)$ respectively, so we can write:
\begin{equation}
\label{equation:sphere}
S^{4p+3} = \mathrm{SU}(2p+2)/\mathrm{SU}(2p+1) = \mathrm{U}(2p+2)/\mathrm{U}(2p+1).
\end{equation}
The unitary frame bundle $F_{\C}S^{4p+3}$ can be identified with $\mathrm{U}(2p+2)$. Thus the subgroup $\mathrm{SU}(2p+2)$ of $ \mathrm{U}(2p+2)$, seen as 
a principal $\mathrm{SU}(2p+1)$-bundle $F_{\C,\mathrm{SU}} S^{4p+3}$ over $S^{4p+3}$ is a reduction of $F_{\C}S^{4p+3}$ to $\mathrm{SU}(2p+1)$.

The aim of this section is to examine the possible further reductions of $F_{\C,\mathrm{SU}} S^{4p+3}$. Note that as far as the spheres $S^{4p+1}$ are concerned (for $p \geq 1$), it was proved in \cite{Leonard-71} that their special unitary frame bundle $F_{\C,\mathrm{SU}} S^{4p+1} \to S^{4p+1}$ does not admit any reduction.

\begin{theorem}
\label{theorem:topology}
Let $p \geq 1$. Assume that the principal $\mathrm{SU}(2p+1)$-bundle $F_{\C,\mathrm{SU}} S^{4p+3}$ over $S^{4p+3}$ admits a reduction of its structure group to a strict connected subgroup $H_0 \lneqq \mathrm{SU}(2p+1)$. Then one of the following holds:
\begin{enumerate}[label=\emph{(\roman*)}]
\item \label{item:1} Either the representation of $H_0$ on $\C^{2p+1}$ is reducible;
\item \label{item:1,5} Or $p=1$, and $H_0$ is contained in $\mathrm{SO}(3) \lneqq \mathrm{SU}(3)$;
\item \label{item:2} Or $p=13$, and $H_0$ is contained in  $\mathrm{E}_6 \lneqq \mathrm{SU}(27)$.
\end{enumerate}
\end{theorem}

%

\begin{proof}
Let $\hat H_0$ be a maximal strict subgroup of $\mathrm{SU}(2p+1)$ containing $H_0$. Clearly the principal bundle $P_0=\mathrm{SU}(2p+2)\to S^{4p+3}$ also reduces to $\hat H_0$. Then \cite[Theorem 3]{Leonard-71} applied (with the notation of \cite{Leonard-71}) to $G_{2p + 1} := \mathrm{SU}(2p + 1)$, shows that $\hat H_0$ is a simple Lie group. 

If $\hat H_0$ is a classical simple Lie group and $p\ge 2$, by \cite[Theorem 2.1, (D) and (E)]{Cadek-Crabb-06} applied to $G=\hat H_0$, we immediately obtain that the representation of $\hat H_0$ (and thus also the one of $H$) on $\C^{2p+1}$ is reducible. The above result does not hold for $p=1$ (when the corresponding sphere $S^7$ is parallelizable), but it is easy to check that the only simple Lie group strictly contained in $\mathrm{SU}(3)$ whose representation on $\C^3$ is irreducible, is $\mathrm{SO}(3)$, embedded in $\mathrm{SU}(3)$ via the complexification of its standard representation on $\R^3$. This corresponds  to case {\em \ref{item:1,5}} of Theorem \ref{theorem:topology}.

 It remains to study the case where $\hat H_0$ is (a finite quotient of) one of the 5 exceptional simple compact Lie groups.

First of all, it suffices to look at complex irreducible representations of the exceptional Lie groups of odd dimension $2p+1$. Moreover, by \cite[Proposition 3.1]{Cadek-Crabb-06}, writing $4p+3=\dim \hat H_0 + k + 1$ for some integer $k$, there must exist at least $k$ vector fields on the sphere $S^{4p+3}$, so the Radon-Hurwitz number $\rho(n)$ (defined by the fact that $\rho(n) - 1$ is the maximal number of linearly independent vector fields on $S^{n-1}$) satisfies 
\begin{equation}\label{rh1}\rho(4p+4)\ge 4p+3-\dim \hat H_0.\end{equation} 
For all $p\ge 3$ we have $\rho(4p+4)\leq 2p+3$. Since no exceptional Lie group has an irreducible complex representation of dimension less than $7$, it follows that $2p+1$ has to be the dimension of an irreducible complex representation of an exceptional Lie group $\hat H_0$, with 
\begin{equation}\label{rh}7\leq 2p+1\leq \dim \hat H_0+1.\end{equation} 
Denoting by  $\widehat{\mathfrak{h}}$ the Lie algebra of $\hat H_0$, it turns out that  there is no complex odd-dimensional irreducible representation of an exceptional Lie group $\hat H_0$ satisfying \eqref{rh} except in the following two cases: \\

\textbf{1.} $\widehat{\mathfrak{h}}=\mathfrak{e}_6$, $\dim \widehat{\mathfrak{h}}=78$. There are two 27-dimensional irreducible representations of $\mathfrak{e}_6$ satisfying \eqref{rh}. This case corresponds to a (theoretical) reduction of the structure group $\mathrm{SU}(27)$ of $F_{\C,\mathrm{SU}}S^{55}$ to a subgroup of $\mathrm{E}_6$ (case \emph{\ref{item:2}} of Theorem \ref{theorem:topology}). \\

\textbf{2.} $\widehat{\mathfrak{h}}=\mathfrak{g}_2$, $\dim  \widehat{\mathfrak{h}}=14$. The only complex odd-dimensional irreducible representation of $\mathfrak{g}_2$ satisfying \eqref{rh} is the complexification of the real $7$-dimensional representation $\rho_7:\mathrm{G}_2\to \mathrm{SO}(7)$ given by the embedding $\mathrm{G}_2\subset \mathrm{SO}(7)$, for $p=3$. However, we will show that $F_{\C,\mathrm{SU}}S^{15}$ does not admit any reduction to $\mathrm{G}_2$. 

Indeed, if such a reduction $P_{\mathrm{G}_2}$ exists, then the tangent bundle $TS^{15}$ is isomorphic to the direct sum $\R\oplus (\C\otimes F_7)$, where $F_7:=P_{\mathrm{G}_2}\times_{\rho_7}\R^7$ is a real vector bundle of rank $7$ over $S^{15}$. 

Now, to each rank $k$ real vector bundle $E$ over $S^{15}$ one can associate an element $\alpha(E)$ in the homotopy group $\pi_{14}(\mathrm{SO}(k))$, namely, the homotopy class of its clutching function at the equator $S^{14} \hookrightarrow S^{15}$. For $k<l$, let $f_{k,l}:\mathrm{SO}(k)\to \mathrm{SO}(l)$ be the standard embedding and denote by $g_{k,l}:\pi_{14}(\mathrm{SO}(k))\to\pi_{14}(\mathrm{SO}(l))$ the group morphisms induced by $f_{k,l}$ in homotopy. If $E$ and $F$ have ranks $k$ and $l$ respectively, we clearly have
\[
\alpha(E\oplus F)=g_{k,k+l}(\alpha(E))+g_{l,k+l}(\alpha(F)).
\]
In our situation, since $\C\otimes F_7$ is topologically isomorphic to $F_7\oplus F_7$, $T S^{15}$ is isomorphic to $\R\oplus F_7\oplus F_7$, so we
can write
\[
\begin{split}
\alpha (T S^{15}) & =g_{14,15}(\alpha(F_7\oplus F_7)) \\
&=g_{14,15}(2g_{7,14}(\alpha(F_7)))=2g_{14,15}(g_{7,14}(\alpha(F_7)))=0,
\end{split}
\]
because $\pi_{14}(\mathrm{SO}(15))=\Z_2$. This is a contradiction since the tangent bundle of $S^{15}$ is non-trivial. Therefore the case $\widehat{\mathfrak{h}}=\mathfrak{g}_2$ is impossible, thus finishing the proof.
\end{proof}

Combined with Theorems \ref{theorem:brin-lefeuvre} and \ref{theorem:non-abelian-livsic}, Theorem \ref{theorem:topology} yields the following:

\begin{corollary}
\label{corollary:invariant-sections}
Let $(M,g,J)$ be a closed connected Kähler manifold with even complex dimension $m$ and non-ergodic unitary frame flow. Then, if $m \neq 4, 28$, there exists a finite cover $(\widehat{M},\widehat{g},\widehat{J})$ of $(M,g,J)$ and a flow-invariant orthogonal projector $\pi_{\mc{V}} \in C^\infty(S\widehat{M},S^2 \mc{N})$ onto a complex subbundle $\mc{V} \subset \mc{N}$ of rank $1 \leq r \leq m/2-1$, of even Fourier degree.
\end{corollary}

Note that by \eqref{equation:n}, $\mc{N}$ is a subbundle of the pullback bundle $\pi^*TM$ so it makes sense to talk about the decomposition of a section $f \in C^\infty(SM, S^2 \mc{N})$ as a sum of spherical harmonics as in \eqref{equation:sum}. The fact that $\mc{V}$ is complex is equivalent to the commutation relation $[\pi_{\mc{V}},J]=0$.

\begin{proof}[Proof of Corollary \ref{corollary:invariant-sections}]

Up to replacing $M$ by a finite covering if necessary, we can assume that the transitivity group $H \leqslant \mathrm{U}(m-1)$ is connected, see \cite[Lemma 3.3]{Cekic-Lefeuvre-Moroianu-Semmelmann-21}. In what follows, in order to keep notation simple, we will still denote this finite cover by $M$.

The representation $\rho : H \to \mathrm{U}(m-1)$ induces a representation $\det \rho : H \to \mathrm{U}(1)$ whose image is either $\mathrm{U}(1)$ or $\left\{1\right\}$ (by connectedness of $H$). Following an argument of Brin-Gromov \cite{Brin-Gromov-80}, we first show that $(\det \rho)(H)= \mathrm{U}(1)$.

Indeed, assume that $(\det \rho)(H) = \left\{1\right\}$. As $(\det \rho)(H)$ is the transitivity group of the frame flow of the complex line bundle $\Lambda^{m-1,0} \mc{N}$, we get by the non-Abelian Liv\v sic Theorem \ref{theorem:non-abelian-livsic} that $\Lambda^{m-1,0} \mc{N}$ is trivial. Now, using that
\[
\Lambda^{m-1,0} \mc{N} \to \Lambda^{m,0} \pi^*TM, \qquad \omega \mapsto \omega \wedge (v-iJv)
\]
is an isomorphism, the triviality of $\Lambda^{m-1,0} \mc{N}$ implies that
\[
c_1(\Lambda^{m,0} \pi^*TM) = \pi^* c_1(\Lambda^{m,0} TM) = 0 \in H^2(SM,\Z).
\]
However, it can be easily checked using the Gysin sequence \cite[Proposition 14.33]{Bott-Tu-82} (and the fact that the dimension of $M$ is $n \geq 4$) that
\[
\pi^* : H^2(M,\Z) \to H^2(SM,\Z)
\]
is injective, so $c_1(\Lambda^{m,0} TM) = 0 = -c_1(K_M)$, where $K_M = \Lambda^{m,0} T^*M$ is the canonical line bundle. This is impossible since $(M,g)$ has negative sectional curvature. Hence, $(\det \rho)(H)= \mathrm{U}(1)$.

%

We will now show that the $H$-representation $\rho$ on $\C^{m-1}$ is reducible. Assume for a contradiction that $\rho$ is irreducible. By the Schur Lemma, the center $C(H)$ of $H$ is contained in the set $\mathrm{U}(1)$ of scalar matrices. On the other hand, the fact that $(\det \rho)(H) = \mathrm{U}(1)$ shows that $H$ is not semi-simple, so its center is at least 1-dimensional. We thus obtain the equality $C(H)=\mathrm{U}(1)$, i.e. $H$  contains the set of scalar matrices.

We now fix an arbitrary point $x_0 \in M$, restrict the unitary frame bundle $F_{\C}M$ to a bundle over $S_{x_0}M$ and identify $S_{x_0}M \simeq S^{4p+3}$. As the structure group of $F_{\C} M$ reduces to $H$ by Theorem \ref{theorem:brin-lefeuvre}, we obtain by restriction to any fiber of $SM\to M$ that the structure group of $F_{\C}S^{4p+3}$ also reduces to $H$, i.e. there exists a principal $H$-bundle $P_H \subset F_{\C}S^{4p+3}$. We claim that $F_{\C}S^{4p+3}$ admits a further reduction to $H_0:=H\cap\mathrm{SU}(2p+1)$. Indeed, the principal $\mathrm{SU}(2p+1)$-bundle $F_{\C,\mathrm{SU}}S^{4p+3}$ is already a reduction of $F_{\C}S^{4p+3}$ to $\mathrm{SU}(2p+1)$ and for every $v\in S^{4p+3}$, the intersection of the fibres $F_{\C,\mathrm{SU}}S^{4p+3}(v)\cap P_H(v)$ is non-empty: if $u\in P_H(v)\subset F_{\C}S^{4p+3}(v)$, there exists $z\in \mathrm{U}(1)$ such that $uz\in F_{\C,\mathrm{SU}}S^{4p+3}(v)$, and since $\mathrm{U}(1)\subset H$ (that is, $H$ contains scalar matrices), we also have $uz\in P_H(v)$, so $uz\in F_{\C,\mathrm{SU}}S^{4p+3}(v)\cap P_H(v)$. It is then straightforward to check that $F_{\C,\mathrm{SU}}S^{4p+3}\cap P_H$ is a principal bundle over $ S^{4p+3}$ with group $H_0$.

As $m \neq 4, 28$ by assumption, we can apply case \emph{\ref{item:1}} of Theorem \ref{theorem:topology} to deduce that $H_0 \lneqq \mathrm{SU}(m-1)$ acts reducibly on $\C^{m-1}$. However, as  $H$ was assumed to act irreducibly on $\C^{m-1}$, $H_0 = H \cap \mathrm{SU}(m-1)$ also acts irreducibly on $\C^{m-1}$ and this is a contradiction. Therefore, $H$ acts reducibly on $\C^{m-1}$.

We can then conclude using the non-Abelian Liv\v sic Theorem \ref{theorem:non-abelian-livsic}:  by the remark after Theorem \ref{theorem:non-abelian-livsic}, there exists a smooth (non-zero) flow-invariant orthogonal projector $\pi_{\mc{V}'} \in C^\infty(SM,S^2 \mc{N})$ onto a flow-invariant smooth complex bundle $\mc{V}' \subset \mc{N}$. Following \cite[Lemma 3.10]{Cekic-Lefeuvre-Moroianu-Semmelmann-21}, one can find a (possibly different) smooth non-zero flow-invariant orthogonal projector $\pi_{\mc{V}} \in C^\infty(SM,S^2 \mc{N})$ of \emph{even Fourier degree} onto a flow-invariant smooth complex bundle $\mc{V} \subset \mc{N}$ of complex rank $1 \leq r \leq m/2-1$.
\end{proof}

\subsection{Complex normal twisted conformal Killing tensors}

\label{ssection:cntckt}

If $m \neq 4,28$ and the unitary frame flow is not ergodic, we know by Corollary \ref{corollary:invariant-sections} that there exists a flow-invariant orthogonal projector $\pi_{\mc{V}} \in C^\infty(SM,S^2\mc{N})$ of even Fourier degree and commuting with $J$. The flow-invariance condition is equivalent to $\X \pi_{\mc{V}} = 0$. A crucial step then, is to show that such a flow-invariant section has \emph{finite Fourier degree}, that is, the decomposition \eqref{equation:sum} only involves a finite number of terms. This is the content of the following:

\begin{lemma}
\label{lemma:degree}
For any operation $\mathfrak{o} : \mathrm{Vect} \to \mathrm{Vect}$, a section $f \in C^\infty(SM,\mathfrak{o}(\mc{N}))$ satisfying $\X f = 0$ has finite Fourier degree.
\end{lemma}

The proof of Lemma \ref{lemma:degree} uses the fact that the sectional curvature of $(M,g)$ is negative and relies on the twisted Pestov identity \eqref{equation:pestov}. Lemma \ref{lemma:degree} was first obtained in \cite[Theorem 4.1]{Guillarmou-Paternain-Salo-Uhlmann-16}, see also \cite[Corollary 4.2]{Cekic-Lefeuvre-Moroianu-Semmelmann-22} for a short self-contained proof. As a consequence, we can decompose
\begin{equation}
\label{equation:piv}
\pi_{\mc{V}} = u_{k} + u_{k-2} + \ldots + u_2 + u_0,
\end{equation}
where $k \geq 0$ is even, $u_i \in C^\infty(M,\Omega_i \otimes S^2 TM)$ and $u_k \neq 0$. Moreover, since $J$ has degree $0$ (that is, it does not depend on the velocity variable $v$), the commutation relation $[\pi_{\mc{V}},J]=0$ yields $[u_i,J] = 0$ for all $i \in \left\{0, \ldots, k\right\}$.

We now set
\[
u := u_k \in C^\infty(M,\Omega_k \otimes S^2 TM),
\]
the spherical harmonic of higher degree in the decomposition \eqref{equation:piv} of $\pi_{\mc{V}}$. Using the mapping property \eqref{equation:split} of $\X$, the equation $\X \pi_{\mc{V}} = 0$ then gives $\X_+ u = 0$. Such a section $u$ is called a \emph{twisted conformal Killing tensor} in the literature. Moreover, since $\iota_v \pi_{\mc{V}} := \pi_{\mc{V}} v = 0$ (because $\mc{V}$ is orthogonal to the span of $v$ and $Jv$) and $\iota_v$ has the mapping properties
\[
\iota_v : \Omega_k \otimes S^2 TM \to (\Omega_{k-1} \otimes TM) \oplus (\Omega_{k+1} \otimes TM),
\]
we obtain that $\iota_v u \in C^\infty(M,\Omega_{k-1} \otimes TM)$ is of degree $k-1$. The same argument also shows that $\iota_{Jv} u \in C^\infty(M,\Omega_{k-1} \otimes TM)$.

Using similarly that $\iota_v \iota_v \pi_{\mc{V}} = \langle \pi_{\mc{V}} v, v \rangle = 0$, a refined algebraic argument allows to show that $\iota_v \iota_v u=\iota_{Jv} \iota_{Jv} u \in C^\infty(M,\Omega_{k-2})$ is of degree $k-2$, see \cite[Lemma 4.2]{Cekic-Lefeuvre-Moroianu-Semmelmann-21} for a proof. Furthermore, we have $\iota_v \iota_{Jv} u = \iota_{Jv} \iota_v u = 0$, using that $u$ is symmetric and $J$ is skew-symmetric, and $[u,J]=0$.

A section $u \in C^\infty(M,\Omega_k \otimes S^2 TM)$ satisfying
\begin{equation}
\label{equation:u1}
\X_+ u = 0, \quad \iota_v u \text{ has degree $k-1$}, \quad \iota_v \iota_v u \text{ has degree $k-2$},
\end{equation}
was called in \cite[Section 4.1]{Cekic-Lefeuvre-Moroianu-Semmelmann-21} a \emph{normal twisted conformal Killing tensor}. (The adjective \emph{normal} refers to the conditions on $\iota_v u$ and $\iota_v \iota_v u$.) Here, the section $u$ satisfies the extra condition $[u,J] = 0$. It is thus worth introducing the following terminology:

\begin{definition}
\label{definition:cntckt}
A section $u \in C^\infty(M,\Omega_k \otimes S^2 TM)$ satisfying \eqref{equation:u1} and $[u,J]=0$ is called a \emph{complex normal twisted conformal Killing tensor}.
\end{definition}


By Corollary \ref{corollary:invariant-sections} and the discussion above, we obtain:

\begin{corollary}\label{corollary:tensor}
Let $(M,g,J)$ be a closed connected Kähler manifold with even complex dimension $m$ and non-ergodic unitary frame flow. Then, if $m \neq 4, 28$, there exists a non-zero complex normal twisted conformal Killing tensor $u$ of even degree $k \geq 2$.
\end{corollary}

\begin{proof}
Corollary \ref{corollary:tensor} follows immediately from Corollary \ref{corollary:invariant-sections} and the above discussion, except for the point that $k \geq 2$ which we now prove. If $k = 0$, then $\pi_{\mc{V}} = u =u_0$ is of degree $0$ and thus $\pi_{\mc{V}}$ can be identified with a section in $C^\infty(M,S^2 TM)$. However, we must also have $\iota_v \pi_{\mc{V}} = 0$ for all $v \in TM$ by \eqref{equation:u1}, so $\pi_{\mc{V}} = 0$, which contradicts the non-vanishing of $\pi_{\mc{V}}$.
\end{proof}

The aim of the remaining sections is now to rule out the existence of such a non-zero complex twisted conformal Killing tensor of even degree $k \geq 2$ under a holomorphic pinching condition $\lambda > \lambda(m)$.

\section{Bounding the terms in the twisted Pestov identity}

\label{section:pestov}

Throughout this section, $(M,g,J)$ is a negatively-curved compact Kähler manifold of real dimension $n=2m$ with $\lambda$-pinched holomorphic curvature, and $u \in C^\infty(M,\Omega_k \otimes S^2 TM)$ is a complex normal twisted conformal Killing tensor of even degree $k \geq 2$. Our aim is to bound from above the terms appearing on the right-hand side of the twisted Pestov identity \eqref{equation:pestov}, namely, the first term $\langle R \nablaV^{S^2 TM} u, \nablaV^{S^2 TM} u \rangle_{L^2}$ and the second term $\langle \mc{F}^{S^2 TM}u, \nabla_{\V}^{S^2 TM}u \rangle_{L^2}$, and to bound from below the term $\|\X_-u\|^2_{L^2}$ on the left-hand side. Sometimes, it will be convenient to consider general vector bundles $E \to M$ rather than the specific bundle $S^2 TM$.

%

In order to simplify notation, we will drop the volume forms in the integrands. The reader should keep in mind that integrals over spheres $S_xM$ (for $x \in M$) are always computed with respect to the round measure $|dv|$ on the sphere, while integrals over $M$ are computed with respect to the Riemannian measure $|dx|$ induced by the metric $g$. Moreover, we will often work with expressions involving a local orthonormal basis of a vector bundle $E$, usually denoted by $(\e_\alpha)_\alpha$; for the simplicity of notation, when we write sums over $\alpha$ we will mean that the sums are pointwise (and the basis might change from point to point). We also introduce the following constants:
\begin{equation}
\label{equation:constants}
\begin{array}{ll}
\alpha_{n,k} := k(n+k-2), & \beta_{n,k} := (k(n+k-2)(n-1))^{1/2},  \\
 \gamma_{n,k} := \dfrac{(n+k-2)(n+2k-4)k}{(n+k-3)(n+2k-2)(k-1)}, & \delta_{n,k} := n+2k-4.
 \end{array}
\end{equation}

\subsection{Bounding the first term in the right-hand side}
Let $E \to M$ be a Euclidean vector bundle equipped with an orthogonal connection $\nabla^E$. Then, the following holds:

\begin{lemma}
\label{lemma:bound-r}
For all $f \in C^\infty(M,\Omega_k \otimes E)$, one has:
\[
\begin{split}
\langle R \nablaV^{E} f, \nablaV^{E} f \rangle_{L^2} \leq - \tfrac{3\lambda-2}{4} & \alpha_{n,k}\|f\|^2_{L^2}  - \tfrac{3\lambda}{4}  \int_{M}\int_{S_xM}\sum_{\alpha} \langle v,J\nabla_{\V}f_\alpha \rangle^2,
\end{split}
\]
where we write locally $f = \sum_{\alpha} f_\alpha \e_\alpha$, for $(\e_\alpha)_{\alpha \in I}$ a local orthonormal basis of $E$.
\end{lemma}

\begin{proof}
Using the upper bound \eqref{equation:sect-pinching0} on the sectional curvature from Lemma \ref{lemma:bishop-goldberg}, we have:
\[
\begin{split}
\langle R \nablaV^{E} f, \nablaV^{E} f \rangle_{L^2} & =  \int_{M} \int_{S_xM}\sum_\alpha R(v,\nablaV^{E} f_\alpha,\nablaV^{E} f_\alpha,v)  \\
& \leq - \tfrac{3\lambda-2}{4} \|\nabla_{\V}^{E} f\|^2_{L^2}  - \tfrac{3\lambda}{4}  \int_{M}\int_{S_xM}\sum_{\alpha} \langle v,J\nabla_{\V}f_\alpha \rangle^2 \\
& = - \tfrac{3\lambda-2}{4} \langle \Delta_{\V}^{E} f, f \rangle_{L^2} - \tfrac{3\lambda}{4} \int_{M}\int_{S_xM}  \sum_{\alpha}\langle v,J\nabla_{\V}f_\alpha \rangle^2  \\
& = - \tfrac{3\lambda-2}{4}k(n+k-2)\|f\|^2_{L^2} - \tfrac{3\lambda}{4}\int_{M}  \int_{S_xM}\sum_{\alpha} \langle v,J\nabla_{\V}f_\alpha \rangle^2 .
\end{split}
\]
Since $\alpha_{n,k} := k(n+k-2)$, this completes the proof.
\end{proof}

\subsection{Bounding the second term in the right-hand side}

Assume now that $E = \Lambda^p TM$ or $E = S^p TM$ for some $p \geq 1$. Using the decomposition of the Riemannian curvature tensor $R = R_0 + \tfrac{1+\lambda}{2}G$ in \eqref{equation:decomp-r}, we can write the second term on the right-hand side of the Pestov identity \eqref{equation:pestov} as:
\begin{equation}
\label{equation:decomp-f}
\langle \mc{F}^{E}f, \nabla_{\V}^{E}f \rangle_{L^2} = \langle \mc{F}^{E}_0f, \nabla_{\V}^{E}f \rangle_{L^2} + \tfrac{1+\lambda}{2} \langle \mc{G}^{E}f, \nabla_{\V}^{E}f \rangle_{L^2}.
\end{equation}
More precisely, $\mc{F}^{E}_0$ and $\mc{G}^{E}$ are defined from $R_0$ and $G$, respectively, by extending the latter to $E$ as in \S \S \ref{sssc:curvature-tensors} and using formula \eqref{equation:f}. We now study separately the two terms in \eqref{equation:decomp-f}. We start with:

\begin{lemma}
\label{lemma:f0}
If $E=\Lambda^pTM$ or $E = S^pTM$, then for all $f \in C^\infty(M,\Omega_k \otimes E)$,
\[
|\langle \mc{F}^{E}_0f, \nabla_{\V}^{E}f \rangle_{L^2}| \leq \tfrac{4p}{3}(1-\lambda) \beta_{n,k} \|f\|^2_{L^2}.
\]
\end{lemma}

Lemma \ref{lemma:f0} will be applied with $E=TM$ and $E=S^2TM$.

\begin{proof}
The proof is the same as \cite[Lemma 4.5]{Cekic-Lefeuvre-Moroianu-Semmelmann-21} by inserting the bound \eqref{equation:bound-r0bar-general}.
\end{proof}

We now study the second term in \eqref{equation:decomp-f}.

\begin{lemma}
\label{lemma:g-expression}
Let $f \in C^\infty(M,\Omega_k \otimes E)$ such that $\iota_v f, \iota_{Jv} f$ are of degree $k-1$. Then, the following holds:
\begin{enumerate}[label=\emph{(\roman*)}]
\item If $E = TM$, one has:
\[
\begin{split}
\langle \mc{G}^{TM}f, \nabla_{\V}^{TM}f \rangle_{L^2} & =  \tfrac{1}{4}\delta_{n,k} \left( \|\iota_v f\|^2_{L^2} + \|\iota_{Jv} f\|^2_{L^2} \right) + \tfrac{1}{2}\|f\|^2_{L^2} \\
& \hspace{2cm} + \tfrac{1}{2}\int_M \int_{S_xM} \sum_\alpha\langle v,J\nabla_{\V} f_\alpha\rangle\langle f,J\e_\alpha\rangle.
\end{split}
\]
\item If $E = S^2 TM$, and $[J,f] = 0$, one has: 
\[
\begin{split}
\langle \mc{G}^{S^2 TM}f, \nabla_{\V}^{S^2 TM}f \rangle_{L^2} =  \delta_{n,k}  \|\iota_v f\|^2_{L^2}  + \|f\|^2_{L^2}.
\end{split}
\]
\end{enumerate}
\end{lemma}


\begin{proof}
We start with the proof for $E = TM$. Note that this equality is an integral equality over $SM$. We will actually prove the integral equality over $S_xM$ for every $x \in M$, and then it suffices to integrate over $x \in M$ to obtain the result. Recall that $G$ is defined in \eqref{eq:complexhyp}. Using the expressions \eqref{equation:f} and \eqref{equation:G}, we have:
	\begin{equation}
	\label{equation:intermediaire0}
	\begin{split}
		4  \sum_\alpha  G(v, \nablaV f_\alpha, f, \e_\alpha) & = \sum_\alpha\big(  \langle (v \wedge \nablaV f_\alpha) f, \e_\alpha \rangle \\
		& + \langle (Jv \wedge J \nablaV f_\alpha) f, \e_\alpha \rangle + 2\langle{v, J \nablaV f_\alpha}\rangle\langle{f, J \e_\alpha}\rangle\big).
		\end{split}
		\end{equation}
		The integral over $S_xM$ of the first term on the right-hand side can be immediately computed using \cite[Lemma 4.6]{Cekic-Lefeuvre-Moroianu-Semmelmann-21} (in the $\Lambda^p$ case with $p=1$)\footnote{We warn the reader that, in the notation of \cite{Cekic-Lefeuvre-Moroianu-Semmelmann-21}, $G$ denotes the curvature tensor of the real hyperbolic space, that is, $G = g \owedge g$. The term $\sum_\alpha \int_{S_xM} \langle v,f\rangle \langle \nablaV f_\alpha, \e_\alpha\rangle - \langle v,\e_\alpha\rangle \langle \nablaV f_\alpha, f\rangle$ thus corresponds exactly to the term computed in \cite[Equation after (4.14)]{Cekic-Lefeuvre-Moroianu-Semmelmann-21} with $p=1$.} since $\iota_v f$ is of degree $k-1$ and yields:
		\[
		\begin{split}
		 \int_{S_xM} \sum_\alpha \langle (v \wedge \nablaV f_\alpha) f, \e_\alpha \rangle& =  \int_{S_xM} \left(\sum_\alpha \langle v,f\rangle \langle \nablaV f_\alpha, \e_\alpha\rangle - \langle v,\e_\alpha\rangle \langle \nablaV f_\alpha, f\rangle\right) \\
		 &= (n+2k-4)\|\iota_v f\|^2_{L^2(S_xM)} + \|f\|^2_{L^2(S_xM)},
		 \end{split}
		\]
		where we use the notation
		\[
		\|f\|^2_{L^2(S_xM)} = \int_{S_xM} g_{x}(f(v),f(v)), \qquad \|\iota_v f\|^2_{L^2(S_xM)} = \int_{S_xM} |(\iota_v f)(v)|^2.
		\]
		
		We claim that the integral over $S_xM$ of the second term in \eqref{equation:intermediaire0} is equal to $(n+2k-4)\|\iota_{Jv} f\|^2_{L^2(S_xM)} + \|f\|^2_{L^2(S_xM)}$.	
		Indeed, observe first that for each $\alpha$ we have 
		\[\langle{(Jv \wedge J\nablaV f_\alpha) f, \e_\alpha}\rangle = \langle{(v \wedge \nablaV f_\alpha)Jf, J\e_\alpha}\rangle.\]
		Now, since $f_\alpha = \langle f,\e_\alpha\rangle = \langle Jf,J\e_\alpha \rangle$, changing the basis $(\e_\alpha)$ by $(J\e_\alpha)$, we see that the second term in  \eqref{equation:intermediaire0} is the same as the first term with $f$ replaced by $Jf$. Using $\iota_v Jf = -\iota_{Jv} f$, the result now follows from the previous computation. Inserting the previous equality in \eqref{equation:intermediaire0} gives the desired result (after integration over $M$) since $\delta_{n,k} = n+2k-4$.
		

We now deal with the case $E = S^2TM$. As before, using the expressions \eqref{equation:f},  \eqref{equation:curv-induced} and \eqref{equation:G}, we have:
	\begin{equation}
	\label{equation:intermediaire}
	\begin{split}
		4 \sum_\alpha G_{S^2 TM}&(v, \nablaV f_\alpha, f, \e_\alpha) = \sum_\alpha \big( \langle [(v \wedge \nablaV f_\alpha), f], \e_\alpha \rangle  \\
		&  + \langle [(J v \wedge J \nablaV f_\alpha), f], \e_\alpha \rangle  - 2 \langle{v, J \nablaV f_\alpha}\rangle.\langle{[J,f], \e_\alpha}\rangle\big),
		\end{split}
		\end{equation}
		where $(\e_\alpha)_{\alpha}$ is a local orthonormal basis of $S^2TM$.
		The integral over $S_xM$ of the first term can be computed using \cite[Lemma 4.6]{Cekic-Lefeuvre-Moroianu-Semmelmann-21} (case of symmetric $2$-tensors) since $\iota_v f$ is of degree $k-1$ and yields:
		\begin{equation}\label{integral}
		 \int_{S_xM} \sum_\alpha\langle [(v \wedge \nablaV f_\alpha), f], \e_\alpha \rangle = 2(n+2k-4)\|\iota_v f\|^2_{L^2(S_xM)} + 2\|f\|^2_{L^2(S_xM)}.
		\end{equation}
		The third term vanishes in \eqref{equation:intermediaire} since $[J,f]=0$ by assumption. 
		
		We now claim that the integral over $S_xM$ of the second term in \eqref{equation:intermediaire} is also equal to $2(n+2k-4)\|\iota_v f\|^2_{L^2(S_xM)} + 2\|f\|^2_{L^2(S_xM)}$, which will finish the proof. 
Indeed, observe that $JX \wedge JY = -J \circ (X \wedge Y) \circ J$ for any $X, Y \in TM$, and therefore using also $[J, f] = 0$, we get for each $\alpha$
		\[[Jv \wedge J \nablaV f_\alpha, f] = -J \circ [v \wedge \nablaV f_\alpha, f] \circ J.\]
		Consequently, we can rewrite
		\[\langle{[Jv \wedge J \nablaV f_\alpha, f], \e_\alpha}\rangle =\langle  -J \circ [v \wedge \nablaV f_\alpha, f] \circ J, \e_\alpha\rangle= \langle{[v \wedge \nablaV f_\alpha, f],- J \circ \e_\alpha \circ  J}\rangle.\]
		Note that $(-J \circ \e_\alpha  \circ J)_\alpha$ is also an orthonormal basis of $S^2TM$, and since $f=-J \circ f \circ J =- \sum_\alpha f_\alpha J \circ \e_\alpha  \circ J$, we obtain
		\[f_\alpha=\langle f,\e_\alpha\rangle=\langle f,-J \circ \e_\alpha \circ  J\rangle.\] 
		The claim thus follows from \eqref{integral}.
\end{proof}

\subsection{Bounding from below the left-hand side}
		
Going back to the case $E = S^2TM$, we now bound from below the term $\|\X_-u\|^2_{L^2}$ appearing on the left-hand side of the twisted Pestov identity \eqref{equation:pestov}:

\begin{lemma}
\label{lemma:x-}
Let $u \in C^\infty(M,\Omega_k \otimes S^2 TM)$ be a complex normal twisted conformal Killing tensor in the sense of Definition \ref{definition:cntckt}. Then,  for $k > 0$ the following inequality holds:
\[
\begin{split}
 \tfrac{(k-1)(n+2k-2)}{k} \|\X_-u\|^2_{L^2} 
\geq& \big(\tfrac{3\lambda-2}{2}\alpha_{n,k-1}-\tfrac{8(1-\lambda)}{3}\beta_{n,k-1}- \tfrac{29(1+\lambda)}{48} - \tfrac{1+\lambda}{4} \delta_{n,k-1}\big) \|\iota_v u\|^2_{L^2}.
\end{split}
\]
\end{lemma}

The proof of Lemma \ref{lemma:x-} requires an additional step and is postponed to the end of this paragraph.

\begin{lemma}
\label{lemma:x+}
Let $f \in C^\infty(M, \Omega_k \otimes TM)$ such that $\iota_v f, \iota_{Jv} f$ are of degree $k-1$, and assume $\lambda \in [\tfrac{2}{3}, 1]$. Then, the following holds:
\[
\begin{split}
\tfrac{k(n+2k)}{k+1}\|\X_+f\|^2_{L^2} \geq &\left( \tfrac{3\lambda-2}{4} \alpha_{n,k} - \tfrac{4(1-\lambda)}{3}  \beta_{n,k} - \tfrac{29(1+\lambda) }{96}\right) \|f\|^2_{L^2} \\
& - \tfrac{1+\lambda}{8}\delta_{n,k}(\|\iota_v f\|^2_{L^2} + \|\iota_{Jv}f\|^2_{L^2}).
\end{split}
\]
\end{lemma}

\begin{proof}
Using the twisted Pestov identity \eqref{equation:pestov} with $E =TM$, and applying the bounds provided by Lemmas \ref{lemma:bound-r},  \ref{lemma:f0} and \ref{lemma:g-expression}, we obtain:
\[
\begin{split}
 \tfrac{k(n+2k)}{k+1}\|\X_+f\|^2_{L^2} 
 \geq& - \langle R \nabla_{\V}^{TM} f, \nabla_{\V}^{TM}f \rangle_{L^2} - \langle \mc{F}^{TM}f,\nabla_{\V}^{TM} f\rangle_{L^2} \\
 \geq&\, \tfrac{3\lambda-2}{4} \alpha_{n,k}\|f\|^2_{L^2} + \tfrac{3\lambda}{4} \int_M\int_{S_xM}  \sum_{\alpha} \langle v,J\nabla_{\V}f_\alpha\rangle^2 -\tfrac{4(1-\lambda)}{3}\beta_{n,k}\|f\|^2_{L^2} \\
& - \tfrac{1+\lambda}{2}\bigg( \dfrac{1}{4}\delta_{n,k}(\|\iota_v f\|^2_{L^2} + \|\iota_{Jv}f\|^2_{L^2}) + \tfrac{1}{2}\|f\|^2_{L^2}\\
&\hskip30pt+ \tfrac{1}{2}\int_M \int_{S_xM} \sum_{\alpha}  |\langle v,J\nabla_{\V}f_\alpha\rangle| |\langle f,J\e_\alpha \rangle| \bigg).
\end{split}
\]
We now use the estimate
\[
\tfrac{1}{2} \int_M \int_{S_xM} \sum_{\alpha} \langle v,J\nabla_{\V}f_\alpha\rangle \langle f,J\e_\alpha \rangle \leq \tfrac{1}{4\eps} \|f\|^2_{L^2} + \tfrac{\eps}{4}\int_M  \int_{S_xM} \sum_\alpha  \langle v,J\nabla_{\V}f_\alpha\rangle^2,
\]
which holds for all $\eps > 0$, to deduce that
\[
\begin{split}
\tfrac{k(n+2k)}{k+1}\|\X_+f\|^2_{L^2}
\geq&\, \tfrac{3\lambda-2}{4} \alpha_{n,k}\|f\|^2_{L^2} + \tfrac{3\lambda}{4}\int_M\int_{S_xM}  \sum_{\alpha}\langle v,J\nabla_{\V}f_\alpha\rangle^2  -\tfrac{4(1-\lambda)}{3}\beta_{n,k}\|f\|^2_{L^2} \\
& - \tfrac{1+\lambda}{2}\bigg( \tfrac{1}{4}\delta_{n,k}(\|\iota_v f\|^2_{L^2} + \|\iota_{Jv}f\|^2_{L^2}) + \left(\tfrac{1}{2} +\tfrac{1}{4\eps}\right)\|f\|^2_{L^2}\\
&\hskip30pt+ \tfrac{\eps}{4} \int_M \int_{S_xM}\sum_{\alpha}  \langle v,J\nabla_{\V}f_\alpha\rangle^2 \bigg).
\end{split}
\]
Taking the specific value $\eps := 6\lambda/(1+\lambda)$, we see that the coefficients in front of the term $\int_M \sum_{\alpha} \int_{S_xM} \langle v,J\nabla_{\V}f_\alpha\rangle^2$ cancel out. Moreover, since by assumption $\lambda \in [2/3,1]$, we have the lower bound $\eps =\tfrac{6\lambda}{1 + \lambda} \geq \tfrac{12}5$ and this eventually yields the result.
\end{proof}
\begin{remark}
	In the last paragraph of the proof of Lemma \ref{lemma:x+}, one could decide to keep using the exact value $\varepsilon = \tfrac{6\lambda}{1 + \lambda}$ in the estimate. However, the benefit of doing so would be minor in the final result so for simplicity we decided to use the trivial lower bound $\varepsilon\geq  \tfrac{12}5$.
\end{remark}

We can now prove Lemma \ref{lemma:x-}.
		
\begin{proof}[Proof of Lemma \ref{lemma:x-}]
Using the equality $\X(\iota_v u) = \iota_v \X u = \iota_v \X_- u$,  and the fact that $\iota_{Jv}\X_-u=J(\iota_{v}\X_-u)$, we obtain:
\[
\begin{split}
\|\X_-u\|^2_{L^2}  & \geq \|\iota_v \X_- u\|^2_{L^2}  + \|\iota_{Jv}\X_-u\|^2_{L^2} =2\|\iota_v \X_- u\|^2_{L^2}  \\
& =2 \|\X(\iota_vu)\|^2_{L^2}  =2 \|\X_+(\iota_v u)\|^2_{L^2}  +2 \|\X_-(\iota_v u)\|^2_{L^2}  \\
& \geq 2\|\X_+(\iota_v u)\|^2_{L^2} ,
\end{split}
\]
where in the second line we used that $\iota_v u$ and $\iota_{Jv} u$ are of degree $k - 1$, and the mapping property \eqref{equation:split}. By assumption, $u$ is a complex normal twisted conformal Killing tensor so this implies that $\iota_v \iota_{Jv} u = \iota_{Jv} \iota_v u = 0$ and $\iota_v \iota_v u= \iota_{Jv}\iota_{Jv}u$ is of degree $k-2$. As a consequence, we can apply Lemma \ref{lemma:x+} with $f:=\iota_v u$  (which is of degree $k-1$). Using the fact that $\iota_{Jv} f=0$, together with the fact that 
\[\|\iota_v f\|^2_{L^2}=\|\iota_v \iota_v u\|^2_{L^2}\leq\|\iota_v  u\|^2_{L^2}, \]
we obtain the announced result.
\end{proof}
		
\section{Pinching estimates} \label{section:threshold}

\subsection{Computations} We start with the following:

\begin{lemma}\label{lemma:sign}
If $u \in C^\infty(M,\Omega_k \otimes S^2 TM)$ is a complex normal twisted conformal Killing tensor, the following inequality holds:
\begin{equation}
\label{equation:sign}
B_{n,k}(\lambda)\|u\|^2_{L^2}  + {C}_{n,k}(\lambda)\|\iota_vu\|^2_{L^2}  \leq 0,
\end{equation}
where 
\begin{equation}
\label{equation:B}
B_{n,k}(\lambda) := \tfrac{3\lambda-2}{4} \alpha_{n,k} - \tfrac{8}{3}(1-\lambda)\beta_{n,k}- \tfrac{1+\lambda}{2},
\end{equation}
and
\begin{equation}
\label{equation:C}
C_{n,k}(\lambda) := \gamma_{n,k}\left( \tfrac{3\lambda-2}{2}\alpha_{n,k-1}-\tfrac{8(1-\lambda)}{3}\beta_{n,k-1}-\tfrac{29(1+\lambda)}{48} - \tfrac{1+\lambda}{4}\delta_{n,k-1}\right) - \tfrac{1+\lambda}{2}\delta_{n,k}.
\end{equation}
\end{lemma}

\begin{proof} Straightforward computation, inserting in the twisted Pestov identity \eqref{equation:pestov} the lower bound for $\|\X_-u\|^2$ (Lemma \ref{lemma:x-}), and the upper bounds for the terms on the right-hand side (Lemmas \ref{lemma:bound-r} and \ref{lemma:f0} applied to $f=u$, and Lemma \ref{lemma:g-expression} (ii)).
\end{proof}
Recall that $n=2m$ is the real dimension of $M$. The end of the proof of Theorem \ref{theorem:main} then consists in finding a pinching condition $\lambda > \lambda(m)$ for which the left-hand side of \eqref{equation:sign} is nonnegative, thus forcing the complex normal twisted conformal Killing tensor $u$ to be zero, which then contradicts Corollary \ref{corollary:invariant-sections}. More precisely, we have the following:
\begin{lemma}\label{lemma:BC}
If the inequalities
\begin{equation}
\label{equation:conditions}
B_{n,k}(\lambda) > 0 \quad \text{and} \quad B_{n,k}(\lambda) + \frac12{C}_{n,k}(\lambda)  > 0
\end{equation}
hold, then \eqref{equation:sign} implies $u \equiv 0$.
\end{lemma}

\begin{proof}
If ${C}_{n,k}(\lambda) \geq 0$, this is immediate from the first part of \eqref{equation:conditions}. If ${C}_{n,k}(\lambda) <0$, using that $ \|u\|^2_{L^2}\geq \|\iota_vu\|^2_{L^2}+\|\iota_{Jv}u\|^2_{L^2}=2\|\iota_vu\|^2_{L^2}$, we get by \eqref{equation:sign}:
\[
(B_{n,k}(\lambda) + \frac12{C}_{n,k}(\lambda))\|u\|^2_{L^2} \leq 0,
\] 
so $u \equiv 0$ by the second part of \eqref{equation:conditions}. 
\end{proof}

Moreover, the following holds:

\begin{lemma}\label{lemma:l12}
One has: $B_{n,k} > 0 \iff \lambda > \lambda_1(n,k)$ where
\[
\lambda_1(n,k) := \dfrac{6\alpha_{n,k}+32\beta_{n,k}+6}{9\alpha_{n,k}+32\beta_{n,k}-6},
\]
and $B_{n,k} +\frac12 C_{n,k} > 0 \iff \lambda > \lambda_2(n,k)$ where
\small
\[
\lambda_2(n,k) :=  \dfrac{6\alpha_{n,k}+32\beta_{n,k}+6+\gamma_{n,k}\left(6\alpha_{n,k-1}+16\beta_{n,k-1}+\tfrac{29}{8}+\tfrac{3}{2}\delta_{n,k-1}\right)+3\delta_{n,k}}{9\alpha_{n,k}+32\beta_{n,k} - 6 +\gamma_{n,k}\left(9\alpha_{n,k-1}+16\beta_{n,k-1}-\tfrac{29}{8}-\tfrac{3}{2}\delta_{n,k-1}\right)-3\delta_{n,k}}.
\]
\normalsize
\end{lemma}

\begin{proof}
Follows immediately from \eqref{equation:B} and \eqref{equation:C}, as the expressions of $B_{n,k}$ and $C_{n,k}$ are affine functions in $\lambda$.
\end{proof}

Before proving Theorem \ref{theorem:main}, we need to study the variations of the sequences $k \mapsto \lambda_{1}(n,k)$ and $k \mapsto \lambda_{2}(n,k)$.

\begin{lemma}
\label{lemma:variation}
For $n \geq 4$, the sequences $k \mapsto \lambda_{1}(n,k)$ and $k \mapsto \lambda_{2}(n,k)$ are decreasing for $k \geq 2$.
\end{lemma}

\begin{proof}
It is straightforward to check that both sequences are positive for $k\ge 2$ and $n\ge 4$. Using that $\alpha_{n,k}=\tfrac{\beta_{n,k}^2}{n-1}$,
we can write 
$$\tfrac{1}{\lambda_1(n,k)}-1= \tfrac{3\alpha_{n,k}-12}{6\alpha_{n,k}+32\beta_{n,k}+6}=\tfrac{3-\tfrac{12}{\alpha_{n,k}}}{6+\tfrac{32(n-1)}{\beta_{n,k}}+\tfrac{6}{\alpha_{n,k}}}.$$
Since  $\alpha_{n,k}$ and $\beta_{n,k}$ are positive and increasing in $k$, the numerator in the right hand term is increasing in $k$, whereas the denominator is positive and decreasing in $k$. Thus $k \mapsto \lambda_{1}(n,k)$ is decreasing.

Consider now the expression $\lambda_{2}(n,k)$. Again it is easy to check that $\lambda_{2}(n,k)>0$ for $k\ge 2$ and $n\ge 4$, and 
\small
\begin{eqnarray*}\tfrac{1}{\lambda_2(n,k)}-1&=&\dfrac{3\alpha_{n,k}-12+\gamma_{n,k}\left(3\alpha_{n,k-1}-\tfrac{29}{4}-3\delta_{n,k-1}\right)-6\delta_{n,k}}{6\alpha_{n,k}+32\beta_{n,k}+6+\gamma_{n,k}\left(6\alpha_{n,k-1}+16\beta_{n,k-1}+\tfrac{29}{8}+\tfrac{3}{2}\delta_{n,k-1}\right)+3\delta_{n,k}}\\
&=&\dfrac{3-\tfrac{12}{\alpha_{n,k}}+\tfrac{\gamma_{n,k}}{\alpha_{n,k}}\left(3\alpha_{n,k-1}-\tfrac{29}{4}-3\delta_{n,k-1}\right)-\tfrac{6\delta_{n,k}}{\alpha_{n,k}}}{6+\tfrac{32\beta_{n,k}}{\alpha_{n,k}}+\tfrac{6}{\alpha_{n,k}}+\tfrac{\gamma_{n,k}}{\alpha_{n,k}}\left(6\alpha_{n,k-1}+16\beta_{n,k-1}+\tfrac{29}{8}+\tfrac{3}{2}\delta_{n,k-1}\right)+\tfrac{3\delta_{n,k}}{\alpha_{n,k}}}.
\end{eqnarray*}
\normalsize
Let us denote this last expression by $\tfrac{E_{n,k}}{F_{n,k}}$. We claim that $E_{n,k}$ is increasing in $k$ and $F_{n,k}$ is decreasing in $k$. 

Using the fact that $\frac{\gamma_{n,k}}{\alpha_{n,k}}=\frac{n+2k-4}{(n+2k-2)\alpha_{n,k-1}}$ we get
\begin{eqnarray*}E_{n,k}&=&3+3\tfrac{\gamma_{n,k}}{\alpha_{n,k}}\alpha_{n,k-1}-\tfrac{12+6\delta_{n,k}}{\alpha_{n,k}}-\tfrac{29}{4}\tfrac{\gamma_{n,k}}{\alpha_{n,k}}-3\tfrac{\gamma_{n,k}}{\alpha_{n,k}}\delta_{n,k-1}\\
&=&6\tfrac{n+2k-3}{n+2k-2}-6\tfrac{n+2k-2}{\alpha_{n,k}}-\tfrac{29}{4}\tfrac{\gamma_{n,k}}{\alpha_{n,k}}-3\tfrac{\gamma_{n,k}}{\alpha_{n,k}}\delta_{n,k-1}\\
&=&6-\tfrac6{n+2k-2}-\tfrac6k-\tfrac6{n+k-2}-3\tfrac{\gamma_{n,k}}{\alpha_{n,k}}\left((n+2k-4)+\tfrac5{12}\right).
\end{eqnarray*}
In order to express $F_{n,k}$, we remark that 
$$\tfrac{\gamma_{n,k}}{\alpha_{n,k}}\beta_{n,k-1}=\tfrac{n+2k-4}{n+2k-2}\tfrac{\beta_{n,k-1}}{\alpha_{n,k-1}}=\tfrac{s_{n,k-1}}{n+2k-2},$$
where we denote by $s_{n,k}:=\frac{(n+2k-2)\beta_{n,k}}{\alpha_{n,k}}.$ A straightforward computation similar to the one for $E_{n,k}$ shows that 
\begin{eqnarray*}F_{n,k}&=&6+3\tfrac{n+2k-2}{\alpha_{n,k}}+6\tfrac{n+2k-4}{n+2k-2}+\tfrac{3}{2}\tfrac{\gamma_{n,k}}{\alpha_{n,k}}\left((n+2k-4)+\tfrac5{12}\right)\\
&&+32\tfrac{s_{n,k}}{n+2k-2}+16\tfrac{s_{n,k-1}}{n+2k-2}.\end{eqnarray*}
In order to prove our claim, it is thus enough to remark that:
\begin{eqnarray*}(n+2k-4)\tfrac{\gamma_{n,k}}{\alpha_{n,k}}&=&\tfrac{(n+2k-4)^2}{(k-1)(n+k-3)(n+2k-2)}\\
&=&\tfrac{n-4}{n-2}\cdot\tfrac1{k-1}+\tfrac{n}{n-2}\cdot\tfrac1{n+k-3}+\tfrac{4}{(k-1)(n+k-3)(n+2k-2)}\end{eqnarray*}
is decreasing in $k$ (and thus $\frac{\gamma_{n,k}}{\alpha_{n,k}}$ is decreasing in $k$ too);
$$s_{n,k}=(n+2k-2)\tfrac{\sqrt{(n-1)k(n+k-2)}}{k(n+k-2)}=\sqrt{(n-1)\left(4+\tfrac{(n-2)^2}{k(n+k-2)}\right)}$$ is decreasing in $k$, and
\begin{eqnarray*}\tfrac{n+2k-2}{\alpha_{n,k}}+2\tfrac{n+2k-4}{n+2k-2}=2+\tfrac{n+2k-2}{k(n+k-2)}-\tfrac4{n+2k-2}=2+\tfrac{(n-2)^2}{k(n+k-2)(n+2k-2)}\end{eqnarray*} is decreasing in $k$.

\end{proof}

\subsection{Proof of ergodicity} We can now conclude the proof of the ergodicity statement in Theorem \ref{theorem:main}.

\begin{proof}[Proof of ergodicity in Theorem \ref{theorem:main}]
We need to show that under the pinching condition $\lambda > \lambda(m)$ of Theorem \ref{theorem:main}, the unitary frame flow is ergodic. If the frame flow is not ergodic, and $m \neq 4,28$, we know by Corollary \ref{corollary:invariant-sections} that there exists a flow-invariant orthogonal projector $\pi_{\mc{V}} \in C^\infty(SM,S^2 \mc{N})$ of even Fourier degree onto a complex vector bundle $\mc{V} \subset \mc{N}$ of rank $1 \leq r \leq m/2-1$.

By Corollary \ref{corollary:tensor} and Lemma \ref{lemma:sign}, this yields the existence of a non-zero complex normal twisted conformal Killing tensor $u \in C^\infty(M,\Omega_k \otimes S^2 T^*M)$ of even degree $k \geq 2$ which satisfies the inequality \eqref{equation:sign}. We distinguish two cases.

{\em Case 1.}  If $k \geq 4$, Lemmas \ref{lemma:l12} and \ref{lemma:variation} show that if the holomorphic pinching $\lambda$ satisfies $\lambda > \max(\lambda_1(n,4),\lambda_2(n,4))$, we then have $B_{n,k} > 0$ and $B_{n,k}+\frac12C_{n,k} > 0$ which by Lemma \ref{lemma:BC} implies that $u \equiv 0$.

{\em Case 2.} If $k=2$, let $\pi_{\mc{V}}$ be the projector given by Corollary \ref{corollary:invariant-sections}, whose complex rank $r$ satisfies $1\le r\le m/2-1$. By \cite[Lemma 4.2]{Cekic-Lefeuvre-Moroianu-Semmelmann-21}, one can then show that $\pi_{\mc{V}}$ must be of the form
\[
\pi_{\mc{V}} = \tfrac{2r}{n} \mathbbm{1}_{TM} + u,
\]
where $u \in C^\infty(M,\Omega_2 \otimes S^2 T^*M)$. In particular, since $\mathbbm{1}_{TM}$ is parallel, this implies $\X u = 0$, that is, $\X_\pm u = 0$. Moreover, a quick algebraic computation (see \cite[p. 38]{Cekic-Lefeuvre-Moroianu-Semmelmann-21}) gives
\begin{equation}
\label{equation:relation-norm}
\|\iota_v u\|^2_{L^2} = \tfrac{2r}{n(n-2r)}\|u\|^2_{L^2},
\end{equation}
and since $2r\le m-2=n/2-2$, we obtain
\begin{equation}
\label{equation:relation-norm1}
\|\iota_v u\|^2_{L^2} \le\tfrac{n-2}{n(n+2)}\|u\|^2_{L^2}.
\end{equation}
Applying the twisted Pestov identity \eqref{equation:pestov} to $u$ for $k=2$, using that $\X_\pm u = 0$ (the bound of Lemma \ref{lemma:x-} is therefore useless), and the upper bounds of Lemmas \ref{lemma:bound-r}, \ref{lemma:f0} and \ref{lemma:g-expression}, we find that:
\[
\left(\tfrac{3\lambda-2}{2}n-\tfrac{8(1-\lambda)}{3}(2n(n-1))^{1/2}-\tfrac{1+\lambda}{2}\right)\|u\|^2_{L^2} - \tfrac{1+\lambda}{2} n \|\iota_vu\|^2_{L^2} \leq 0.
\]
Inserting \eqref{equation:relation-norm} in the previous inequality, we then obtain:
\[
\left(\tfrac{3\lambda-2}{2}n-\tfrac{8}{3}(1-\lambda)(2n(n-1))^{1/2}-\tfrac{1+\lambda}{2} \tfrac{n-2}{n+2} \right)\|u\|^2_{L^2} \leq 0.
\]
This shows that $u \equiv 0$  when $k=2$, as soon as the holomorphic pinching $\lambda$ satisfies
\[
\lambda \geq \lambda_3(n) :=  \dfrac{6n+ 16(2n(n-1))^{1/2} + \tfrac{6n}{n+2}}{9n+16(2n(n-1))^{1/2}-\tfrac{6n}{n+2}}.
\]

Summarizing the two cases $k\ge 4$ and $k=2$, we have obtained that the the unitary frame flow is ergodic as soon as the holomorphic pinching $\lambda$ satisfies
\[
\lambda> \max(\lambda_1(2m,4),\lambda_2(2m,4),\lambda_3(2m))=:\lambda_0(m).
\]
Using a formal computing tool, it can be easily checked that $ \lambda_0(m)=\lambda_2(2m)$ for $m \geq 6$. However, for practical reasons we will give a bound which is slightly less accurate, but much easier to compute. 

Namely, using the obvious inequality $((n+2)(n-1))^{1/2}<n+1/2$, we can write
\begin{eqnarray*}\dfrac{1}{\lambda_1(n,4)}-1&=& \dfrac{3\alpha_{n,4}-12}{6\alpha_{n,4}+32\beta_{n,4}+6}= \dfrac{12(n+2)-12}{24(n+2)+32(4(n+2)(n-1))^{1/2}+6}\\
&>&\dfrac{12n+12}{88n+86}=\dfrac{6n+6}{44n+43}.
\end{eqnarray*}
Similarly, since $16\sqrt2<68/3$, we have
\begin{eqnarray*}\dfrac{1}{\lambda_3(n)}-1&=&  \dfrac{3n-\tfrac{12n}{n+2}}{6n+ 16(2n(n-1))^{1/2} + \tfrac{6n}{n+2}}=  \dfrac{3-\tfrac{12}{n+2}}{6+ 16(\tfrac{2(n-1)}{n})^{1/2} + \tfrac{6}{n+2}}\\
&>& \dfrac{3-\tfrac{12}{n+2}}{6+ \tfrac{68}3 + \tfrac{6}{n+2}}=\dfrac{9n-18}{86 n+190}.
\end{eqnarray*}
Finally, using the calculation in the proof of Lemma \ref{lemma:variation} together with the obvious inequalities $64/\sqrt3<37$ and 
$$\dfrac43>\gamma_{n,4}= \dfrac{4(n+2)(n+4)}{3(n+1)(n+6)}> \dfrac{4n}{3(n+1)},\qquad\forall n\ge3,$$
we obtain 
\small
\begin{eqnarray*}&&\hskip-20pt\tfrac{1}{\lambda_2(n,4)}-1=\dfrac{3\alpha_{n,4}-12+\gamma_{n,4}\left(3\alpha_{n,3}-\tfrac{29}{4}-3\delta_{n,3}\right)-6\delta_{n,4}}{6\alpha_{n,4}+32\beta_{n,4}+6+\gamma_{n,4}\left(6\alpha_{n,3}+16\beta_{n,3}+\tfrac{29}{8}+\tfrac{3}{2}\delta_{n,3}\right)+3\delta_{n,4}}\\
&>&\dfrac{12(n+2)-12+\tfrac{4n}{3(n+1)}\left(9(n+1)-\tfrac{29}{4}-3(n+2)\right)-6(n+4)}{24(n+2)+64(n+\tfrac12)+6+\tfrac43\left(18(n+1)+16\sqrt3n+\tfrac{29}{8}+\tfrac{3}{2}(n+2)\right)+3(n+4)}\\
&>&\dfrac{14n-\tfrac{77}3}{(117+\tfrac{64}{\sqrt3})n+126+\tfrac{29}{6}}>\dfrac{14n-26}{154n+131}.
\end{eqnarray*}
\normalsize
It is straightforward to check that
$$\dfrac{6n+6}{44n+43}>\dfrac{9n-18}{86 n+190}>\dfrac{14n-26}{154n+131},\qquad\forall n\ge 10,$$
which implies $\tfrac1{\lambda_0(m)}-1>\dfrac{14n-26}{154n+131}$, so eventually 
$$\lambda_0(m)<\dfrac{154n+131}{168n+105}=\dfrac{308m+131}{336m+105}=:\lambda(m),$$
thus proving the ergodicity statement of Theorem \ref{theorem:main}.
\end{proof}

We conclude this paragraph by a remark on the remaining cases $m=4,28$:

\begin{remark} In the case $m=28$, one can check\footnote{Using the LiE program for instance.} that $\mathrm{E}_6$ fixes an element of $S^3 \C^{27}$. As we saw in the proof of Corollary \ref{corollary:invariant-sections}, the transitivity group is \emph{never} semi-simple as it always contains the scalar matrices: Theorem \ref{theorem:topology} therefore does not exclude that subgroups of $\mathrm{E}_6 \times \mathrm{U}(1) \lneqq \mathrm{U}(27)$ occur as transitivity groups in complex dimension $m=28$. In this case, the transitivity group would fix an orthogonal projector of $\mathrm{End}(S^3 \C^{27})$. Nevertheless, the argument given below involving the twisted Pestov identity does not carry over to $\mathrm{End}(S^3 \C^{27})$ because this vector space involves tensorial powers of too high degree. In other words, the argument would result in a pinching condition $\lambda > \lambda(28)$ for some $\lambda(28) > 1$ so the statement would be empty. A similar remark holds for the case $m=4$. This should be compared with \cite[Theorem 3.8]{Cekic-Lefeuvre-Moroianu-Semmelmann-21} where we only deal with elements of $\mathfrak{o}(\mc{N})$ with $\mathfrak{o} \in \left\{\mathrm{id},S^2,\Lambda^2,\Lambda^3\right\}$. The worse pinching estimate in \cite[Theorem 1.2]{Cekic-Lefeuvre-Moroianu-Semmelmann-21} comes from the exterior power $\Lambda^3 \mc{N}$.

\end{remark}

\subsection{Proof of mixing} \label{section:5.3}
It now remains to show the mixing property. We will actually show the following:

\begin{proposition}
\label{proposition:mixing}
The unitary frame flow on a negatively-curved Kähler manifold $(M,g)$ is ergodic if and only if it is mixing.
\end{proposition}

\begin{proof}
Mixing implies ergodicity so it remains to show that ergodicity implies mixing in this setting. By \cite[Proof of Lemma 3.7]{Lefeuvre-21}, it suffices to show that the equation
\begin{equation}
\label{equation:mixing}
X_{F_{\C}M} u = i \lambda u, \qquad \lambda \in \R \setminus \left\{0\right\},\  u \in L^2(F_{\C}M),
\end{equation}
implies that $u \equiv 0$. Let $u \in L^2(F_{\C}M)$ be a solution to \eqref{equation:mixing}. Since $F_{\C}M \to SM$ is a principal $\mathrm{U}(m-1)$-bundle over $SM$, the space $L^2(F_{\C}M)$ splits as
\begin{equation}
\label{equation:splitting}
L^2(F_{\C}M) = \bigoplus_{k=0}^{+\infty} L^2(SM, \Theta_k),
\end{equation}
where $\Theta_k$ is the vector bundle over $SM$ whose fibre over $v\in SM$ is the eigenspace of the Casimir operator of $\mathrm{U}(m-1)$ acting on functions of the fibre $(F_{\C}M)_v$, associated to the eigenvalue $\mu_k \geq 0$ (and $\mu_k \neq \mu_j$ for $k \neq j$). Since $X_{F_{\C}M}$ preserves the splitting \eqref{equation:splitting} (because the frame flow $(\Phi_t)_{t \in \R}$ commutes with the right-action of $\mathrm{U}(m-1)$), we deduce that for each $k\ge 0$, the component $u_k\in L^2(SM,\Theta_k)$ of some function $u = \sum_{k} u_k \in \oplus_{k=0}^{+\infty}L^2(SM,\Theta_k)$ satisfying \eqref{equation:mixing} must satisfy
\[
X_{F_{\C}M} u_k = i \lambda u_k.
\]

Since $\Theta_0 = \C$ is trivial over $SM$, the equation $X_{F_{\C}M} u_0 = i\lambda u_0$ reads $X u_0 = i \lambda u_0$, where $u_0 \in L^2(SM)$ and $X$ is the geodesic vector field. But the geodesic flow is mixing as $(M,g)$ has negative curvature, so we deduce that $u_0 = 0$.

For $k \neq 0$, \cite[Proposition 2.6]{Lefeuvre-21} based on microlocal analysis shows that $u_k \in L^2(SM,\Theta_k)$ is actually smooth, that is, $u_k \in C^\infty(SM,\Theta_k)$. Moreover, $X_{F_{\C}M}|u_k|^2 = 0$, so using the ergodicity assumption on the frame flow $(\Phi_t)_{t \in \R}$, we get that $|u_k|$ is constant on $F_{\C}M$. If $u_k \neq 0$, then up to rescaling, we obtain that $u_k : F_{\C} M \to \mathrm{U}(1)$ is a well-defined smooth map. In order to simplify the notation, we will drop the index $k$ from now on and simply write $u$ instead of $u_k$, and $\mu$ instead of $\mu_k$.

Denote by $P$ the unit circle bundle of the complex line bundle $\Lambda^{m-1,0}\mc{N} \to SM$. Now, observe that there is a natural surjective bundle map $\psi : F_{\C}M \to P$ given by
\[
\psi : (v, \e_2, \ldots,\e_m) \mapsto (v, (\e_2-i J\e_2) \wedge \ldots\wedge (\e_m-iJ\e_m)).
\]
There is a natural unitary parallel transport along geodesic flow-lines of sections of $\Lambda^{m-1,0}\mc{N}$, so there is a flow $(\Phi_t^P)_{t \in \R}$ on $P$ with generator $X_P$ extending the geodesic flow $(\varphi_t)_{t\in\R}$ as in \S\ref{ssection:isometry}, that is, writing $\pi : P \to SM$ for the projection map, one has $\pi \circ \Phi_t^P = \varphi_t \circ \pi$ for all $t \in \R$. Moreover, $\psi$ intertwines the frame flow on $F_{\C}M$ and the flow on $P$, that is
\begin{equation}
\label{equation:relation}
\Phi_t^P \circ \psi = \psi \circ \Phi_t.
\end{equation}
We claim that the following holds:

\begin{claim}
There exists a smooth function $w \in C^\infty(P)$ such that $u = \psi^* w$, and $X_P w = i \lambda w$.
\end{claim}

Note that, once we know that $u = \psi^*w$ for some function $w$, the relation $X_P w = i \lambda w$ is immediate using \eqref{equation:relation} and $X_{F_{\C}M} u = i \lambda u$. 

\begin{proof}
We show that $u = \psi^*w$ for some $w \in C^\infty(P)$. We fix an arbitrary point $v_0 \in SM$ and $w_0 \in (F_{\C}M)_{v_0}$. There is then a commutative diagram
\begin{equation}
\label{equation:diagram}
\xymatrix{
    (F_{\C}M)_{v_0} \ar[r]^\psi \ar[d]  & P_{v_0} \ar[d] \\
    \mathrm{U}(m-1) \ar[r]^{\mathrm{det}} & \mathrm{U}(1)
  }
\end{equation}
where the downward arrows are isometries. Hence, by restricting $u$ to the fiber $(F_{\C}M)_{v_0}$ and identifying isometrically $(F_{\C}M)_{v_0} \simeq \mathrm{U}(m-1)$, we get a map $F := u(v_0)$ such that
\begin{equation}
\label{equation:submersion}
F  : F_{\C}M \simeq \mathrm{U}(m-1) \to \mathrm{U}(1),
\end{equation}
and $\Delta_{\mathrm{U}(m-1)}F = \mu F$ with $\mu \neq 0$ since $u$ takes values in $\Theta_k$ for $k \neq 0$. In other words, $F$ is an eigenfunction of the Casimir operator on $\mathrm{U}(m-1)$ associated to the eigenvalue $\mu \neq 0$ and of constant modulus.

Now, recall that $\mathrm{U}(m-1)$ is a split group extension of the circle group $\mathrm{U}(1)$ by the special unitary group $\mathrm{SU}(m-1)$, that is, $\mathrm{det} : \mathrm{U}(m-1) \to \mathrm{U}(1)$ is a fiber bundle with fibers isometric to $\mathrm{SU}(m-1)$. We claim that such a function $F$ must then necessarily be constant on the $\mathrm{SU}(m-1)$-fibers of the bundle $\mathrm{U}(m-1)$, that is, $F = \mathrm{det}^* f$ for some $f \in C^\infty(\mathrm{U}(1))$, $\mu = j^2$ for some integer $j \geq 0$ and $f$ is an eigenfunction of $\Delta_{\mathrm{U}(1)}$ with eigenvalue $j^2$. 

Indeed, since $\mathrm{U}(m-1) = \mathrm{SU}(m-1) \times_{\Z_m} \mathrm{U}(1)$, there is a $\Z_m$-bundle map\footnote{The $\Z_m$-action on $\mathrm{SU}(m-1) \times \mathrm{U}(1)$ is simply given by $(w,z) \mapsto (w \omega^{-k}, \omega^k z)$ for $(w,z) \in \mathrm{SU}(m-1) \times \mathrm{U}(1)$, $k \in \left\{0,...,m-2\right\}$ and $\omega := e^{2i\pi/{(m-1)}}$.} $p : \mathrm{SU}(m-1) \times \mathrm{U}(1) \mapsto \mathrm{U}(m-1)$ and this map is locally a Riemannian isometry. As a consequence, we get that $\Delta_{\mathrm{SU}(m-1) \times \mathrm{U}(1)} p^*F = \mu~ p^* F$. As $\mathrm{SU}(m-1) \times \mathrm{U}(1)$ is a Riemannian product, its eigenfunctions are obtained as sums of products of eigenfunctions on each factor and the eigenvalues are sums of the eigenvalues on each factor. Hence, we can write
\[
p^*F(w,z) = \sum_{j=1}^N a_j(w) b_j(z) = \sum_{j=1}^N a_j(w) z^{k_j},
\]
for some finite number $N > 0$, where $(w,z) \in \mathrm{SU}(m-1) \times \mathrm{U}(1)$, $k_j \in \Z$ and $k_{j_1} \neq k_{j_2}$ for $j_1 \neq j_2$, $a_j \neq 0$ is an eigenfunction of $\Delta_{\mathrm{SU}(m-1)}$ associated to the eigenvalue $\lambda_j$ and $\lambda_j + k^2_j = \mu > 0$. However, we also know that $p^*F$ has constant modulus (equal to $1$). Freezing the point $w = \mathbf{1}_{\C^{m-1}} \in \mathrm{SU}(m-1)$ and moving $z \in \mathrm{U}(1)$, we easily get that $N$ must be equal to $1$, that is, $p^*F(w,z) = a(w) z^{k}$ for some $k\in\Z$. Hence $a \in C^\infty(\mathrm{SU}(m-1))$ satisfies $|a|=1$ and $\Delta_{\mathrm{SU}(m-1)} a = \lambda a$ for some $\lambda \geq 0$ (with $k^2+\lambda=\mu$).

We claim that $\lambda = 0$, that is, $a$ is constant, and $F$ is thus constant along the $\mathrm{SU}(m-1)$-fibers. Indeed, if $\lambda \neq 0$, using 
\[
\begin{split}
0&=\Delta_{\mathrm{SU}(m-1)} |a|^2\\
& = (\Delta_{\mathrm{SU}(m-1)} a) \overline{a} + a (\Delta_{\mathrm{SU}(m-1)} \overline{a}) - 2 \nabla a \cdot \nabla \overline{a} \\
& = 2 \lambda - 2 |\nabla a|^2,
\end{split}
\]
we get that $|\nabla a|$ is a non-zero constant. Since $\mathrm{SU}(m-1)$ is simply connected, the map $a$ lifts to the universal cover of $\mathrm{U}(1)$, so there exists a map $\theta:\mathrm{SU}(m-1)\to \R$ such that $a=e^{i\theta}$. At a critical point of $\theta$ we thus get $\nabla a=0$, which is absurd.

As a consequence, $\lambda = 0$ and this shows that $F$ is constant along $\mathrm{SU}(m-1)$-fibers of the bundle map $\mathrm{det} : \mathrm{U}(m-1) \to \mathrm{U}(1)$. In turn, as \eqref{equation:diagram} commutes, we get that $u$ is constant on the preimages of the map $\psi : F_{\C} M \to P$ so $u = \psi^* w$ for some smooth $w \in C^\infty(P)$.
\end{proof}

As a consequence, we have just shown that if there is a non-zero solution to \eqref{equation:mixing}, then there is also a non-zero solution to
\begin{equation}
\label{equation:mixing2}
X_P w = i \lambda w, \qquad \lambda \in \R \setminus \left\{0\right\}, w \in C^\infty(P).
\end{equation}
Hence, it remains to show that \eqref{equation:mixing2} has no non-zero solutions, and we will actually show it for $w \in L^2(P)$.


Assume that $w \in L^2(P)$ is a solution to \eqref{equation:mixing2}. As above, using that $P$ is a principal $\mathrm{U}(1)$-bundle, we can decompose
\begin{equation}
\label{equation:l2split}
L^2(P) = \bigoplus_{k \in \Z} L^2(SM, L_k),
\end{equation}
where $L_k$ is the complex line bundle given by functions $f \in L^2(P)$ satisfying $V f = ik f$, where $V$ is the infinitesimal generator of the $\mathrm{U}(1)$-action on the fibers of $P$. Observe that $L_k = L^{\otimes k}$ for $k \in \Z$ where $L := L_1$ and $L_1 \simeq \Lambda^{m-1,0} \mc{N}$. Now, using the splitting \eqref{equation:l2split}, the equation $Xw = i \lambda w$ reads $X w_k = i \lambda w_k$ for all $k \in \Z$, where $w_k \in L^2(SM,L_k)$. By \cite[Proposition 2.6]{Lefeuvre-21}, we also have that $w_k$ is smooth, that is, $w_k \in C^\infty(SM,L_k)$. Moreover, $X|w_k|^2 = 0$, so $|w_k|$ is constant by ergodicity of the flow $(\varphi_t)_{t \in \R}$. If $k=0$, $L_0 = \C$ is trivial over $SM$ so the mixing of the geodesic flow $(\varphi_t)_{t \in \R}$ then implies that $w_0 \equiv 0$. If $k \neq 0$, and $w_k \neq 0$, then we obtain a smooth nowhere vanishing section $w_k \in C^\infty(SM,L_k)$ so $L_k$ is trivial. But in turn, it implies that $L_1 = \Lambda^{m-1,0}\mc{N}$ is trivial, which is a contradiction (see the proof of Corollary \ref{corollary:invariant-sections}). Hence, any solution to \eqref{equation:mixing2} is trivial, and thus, so is any solution to \eqref{equation:mixing}. This finishes the proof of Proposition \ref{proposition:mixing}.
\end{proof}

%
%
%

\bibliographystyle{alpha}

\bibliography{Biblio}

\begin{thebibliography}{CLMS22}

\bibitem[Ami86]{Amirov-86}
Arif Amirov.
\newblock Existence and uniqueness theorems for the solution of an inverse
  problem for the transfer equation.
\newblock {\em Sibirsk. Mat. Zh.}, 27(6):3--20, 1986.

\bibitem[Ano67]{Anosov-67}
Dmitri~Viktorovitch Anosov.
\newblock Geodesic flows on closed {R}iemannian manifolds of negative
  curvature.
\newblock {\em Trudy Mat. Inst. Steklov.}, 90:209, 1967.

\bibitem[Ber60a]{Berger-60-1}
Marcel Berger.
\newblock Pincement riemannien et pincement holomorphe.
\newblock {\em Ann. Scuola Norm. Sup. Pisa Cl. Sci. (3)}, 14:151--159, 1960.

\bibitem[Ber60b]{Berger-60-2}
Marcel Berger.
\newblock Sur quelques vari\'{e}t\'{e}s riemanniennes suffisamment pinc\'{e}es.
\newblock {\em Bull. Soc. Math. France}, 88:57--71, 1960.

\bibitem[BG63]{Bishop-Goldberg-63}
Richard~L. Bishop and Samuel~I. Goldberg.
\newblock On the topology of positively curved {K}aehler manifolds.
\newblock {\em Tohoku Math. J. (2)}, 15:359--364, 1963.

\bibitem[BG80]{Brin-Gromov-80}
Michael Brin and Mikhael Gromov.
\newblock On the ergodicity of frame flows.
\newblock {\em Invent. Math.}, 60(1):1--7, 1980.

\bibitem[BK78]{Bourguignon-Karcher-78}
Jean-Pierre Bourguignon and Hermann Karcher.
\newblock Curvature operators: pinching estimates and geometric examples.
\newblock {\em Ann. Sci. \'{E}cole Norm. Sup. (4)}, 11(1):71--92, 1978.

\bibitem[BK84]{Brin-Karcher-84}
Michael Brin and Hermann Karcher.
\newblock Frame flows on manifolds with pinched negative curvature.
\newblock {\em Compositio Math.}, 52(3):275--297, 1984.

\bibitem[BP74]{Brin-Pesin-74}
Michael~I. Brin and Yakov~B. Pesin.
\newblock Partially hyperbolic dynamical systems.
\newblock {\em Izv. Akad. Nauk SSSR Ser. Mat.}, 38:170--212, 1974.

\bibitem[BP03]{Burns-Pollicott-03}
Keith Burns and Mark Pollicott.
\newblock Stable ergodicity and frame flows.
\newblock {\em Geom. Dedicata}, 98:189--210, 2003.

\bibitem[Bri75a]{Brin-75-2}
Michael Brin.
\newblock Topological transitivity of a certain class of dynamical systems, and
  flows of frames on manifolds of negative curvature.
\newblock {\em Funkcional. Anal. i Prilo\v{z}en.}, 9(1):9--19, 1975.

\bibitem[Bri75b]{Brin-75-1}
Michael Brin.
\newblock The topology of group extensions of {$C$}-systems.
\newblock {\em Mat. Zametki}, 18(3):453--465, 1975.

\bibitem[Bri82]{Brin-82}
Michael Brin.
\newblock Ergodic theory of frame flows.
\newblock In {\em Ergodic theory and dynamical systems, {II} ({C}ollege {P}ark,
  {M}d., 1979/1980)}, volume~21 of {\em Progr. Math.}, pages 163--183.
  Birkh\"{a}user, Boston, Mass., 1982.

\bibitem[BT82]{Bott-Tu-82}
Raoul Bott and Loring~W. Tu.
\newblock {\em Differential forms in algebraic topology}, volume~82 of {\em
  Graduate Texts in Mathematics}.
\newblock Springer-Verlag, New York-Berlin, 1982.

\bibitem[{\v{C}}C06]{Cadek-Crabb-06}
Martin {\v{C}}adek and Michael Crabb.
\newblock {$G$}-structures on spheres.
\newblock {\em Proc. London Math. Soc. (3)}, 93(3):791--816, 2006.

\bibitem[CL21]{Cekic-Lefeuvre-21-1}
Mihajlo {Ceki{\'c}} and Thibault {Lefeuvre}.
\newblock {The Holonomy Inverse Problem}.
\newblock {\em arXiv:2105.06376}, May 2021.

\bibitem[CLMS21]{Cekic-Lefeuvre-Moroianu-Semmelmann-21}
Mihajlo {Ceki{\'c}}, Thibault {Lefeuvre}, Andrei {Moroianu}, and Uwe
  {Semmelmann}.
\newblock {On the ergodicity of the frame flow on even-dimensional manifolds}.
\newblock {\em arXiv:2111.14811}, November 2021.

\bibitem[CLMS22]{Cekic-Lefeuvre-Moroianu-Semmelmann-22}
Mihajlo Ceki\'{c}, Thibault Lefeuvre, Andrei Moroianu, and Uwe Semmelmann.
\newblock Towards {B}rin's conjecture on frame flow ergodicity: new progress
  and perspectives.
\newblock {\em Math. Res. Rep.}, 3:21--34, 2022.

\bibitem[GPSU16]{Guillarmou-Paternain-Salo-Uhlmann-16}
Colin Guillarmou, Gabriel~P. Paternain, Mikko Salo, and Gunther Uhlmann.
\newblock The {X}-ray transform for connections in negative curvature.
\newblock {\em Comm. Math. Phys.}, 343(1):83--127, 2016.

\bibitem[HM79]{Howe-Moore-79}
Roger~E. Howe and Calvin~C. Moore.
\newblock Asymptotic properties of unitary representations.
\newblock {\em J. Functional Analysis}, 32(1):72--96, 1979.

\bibitem[Hop36]{Hopf-36}
Eberhard Hopf.
\newblock Fuchsian groups and ergodic theory.
\newblock {\em Trans. Amer. Math. Soc.}, 39(2):299--314, 1936.

\bibitem[HP06]{Hasselblatt-Pesin-06}
Boris Hasselblatt and Yakov Pesin.
\newblock Partially hyperbolic dynamical systems.
\newblock In {\em Handbook of dynamical systems. {V}ol. 1{B}}, pages 1--55.
  Elsevier B. V., Amsterdam, 2006.

\bibitem[JS07]{Jakobson-Strohmaier-07}
Dmitry Jakobson and Alexander Strohmaier.
\newblock High energy limits of {L}aplace-type and {D}irac-type eigenfunctions
  and frame flows.
\newblock {\em Comm. Math. Phys.}, 270(3):813--833, 2007.

\bibitem[JSZ08]{Jakobson-Strohmaier-Zelditch-08}
Dmitry Jakobson, Alexander Strohmaier, and Steve Zelditch.
\newblock On the spectrum of geometric operators on {K}\"{a}hler manifolds.
\newblock {\em J. Mod. Dyn.}, 2(4):701--718, 2008.

\bibitem[KN96]{Kobayashi-Nomizu-96}
Shoshichi Kobayashi and Katsumi Nomizu.
\newblock {\em Foundations of differential geometry. {V}ol. {I}}.
\newblock Wiley Classics Library. John Wiley \& Sons, Inc., New York, 1996.
\newblock Reprint of the 1963 original, A Wiley-Interscience Publication.

\bibitem[{Lef}21]{Lefeuvre-21}
Thibault {Lefeuvre}.
\newblock {Isometric extensions of Anosov flows via microlocal analysis}.
\newblock {\em arXiv:2112.05979}, December 2021.

\bibitem[Leo71]{Leonard-71}
Peter Leonard.
\newblock {$G$}-structures on spheres.
\newblock {\em Trans. Amer. Math. Soc.}, 157:311--327, 1971.

\bibitem[Muk75]{Mukhometov-75}
Ravil~Galatdinovich Mukhometov.
\newblock Inverse kinematic problem of seismic on the plane.
\newblock {\em Akad. Nauk. SSSR}, 6:243--252, 1975.

\bibitem[Muk81]{Mukhometov-81}
Ravil~Galatdinovich Mukhometov.
\newblock On a problem of reconstructing {R}iemannian metrics.
\newblock {\em Sibirsk. Mat. Zh.}, 22(3):119--135, 237, 1981.

\bibitem[Pat99]{Paternain-99}
Gabriel~P. Paternain.
\newblock {\em Geodesic flows}, volume 180 of {\em Progress in Mathematics}.
\newblock Birkh\"{a}user Boston, Inc., Boston, MA, 1999.

\bibitem[PS88]{Pestov-Sharafutdinov-88}
Leonid~N. Pestov and Vladimir~A. Sharafutdinov.
\newblock Integral geometry of tensor fields on a manifold of negative
  curvature.
\newblock {\em Sibirsk. Mat. Zh.}, 29(3):114--130, 221, 1988.

\bibitem[PSU15]{Paternain-Salo-Uhlmann-15}
Gabriel~P. Paternain, Mikko Salo, and Gunther Uhlmann.
\newblock Invariant distributions, {B}eurling transforms and tensor tomography
  in higher dimensions.
\newblock {\em Math. Ann.}, 363(1-2):305--362, 2015.

\bibitem[Sha94]{Sharafutdinov-94}
Vladimir~A. Sharafutdinov.
\newblock {\em Integral geometry of tensor fields}.
\newblock Inverse and Ill-posed Problems Series. VSP, Utrecht, 1994.

\end{thebibliography}

\end{document}